\documentclass{tran-l}

\usepackage{mathtools,amssymb}                                              % Std packages
\usepackage{hyperref}
\usepackage{kbordermatrix}									                % Fancy matrices
\usepackage[usenames,dvipsnames,svgnames,table]{xcolor}
\usepackage{enumitem}
\usepackage{float}
\usepackage{multirow}
\usepackage{physics}
\usepackage{theoremref}
\usepackage{mathrsfs}
\usepackage{graphicx}

\newtheorem{theorem}{Theorem}[section]
\newtheorem{lemma}[theorem]{Lemma}
\newtheorem{corollary}[theorem]{Corollary}
\newtheorem{proposition}[theorem]{Proposition}
\newtheorem{observation}[theorem]{Observation}
\newtheorem{conjecture}[theorem]{Conjecture}

\newtheorem{question}[theorem]{Question}

\theoremstyle{definition}
\newtheorem{definition}[theorem]{Definition}
\newtheorem{example}[theorem]{Example}

\theoremstyle{remark}
\newtheorem{remark}[theorem]{Remark}

\newcommand{\ii}{\mathsf{i}}									            % \textup{i} \coloneqq \sqrt{-1}
\newcommand{\bracket}[1]{\left[#1 \right]}						            % [n]
\newcommand\floor[1]{\left\lfloor#1\right\rfloor}				            % floor function

\DeclareMathOperator{\conv}{conv}
\DeclareMathOperator{\coni}{coni}
\DeclareMathOperator{\diag}{diag}

\DeclareMathOperator{\spec}{spec}

\DeclareMathOperator{\circulant}{circ}
\DeclareMathOperator{\linearspan}{span}
\numberwithin{equation}{section}

\begin{document}

\title[Perron similarities and the NIEP]{Perron similarities and the nonnegative inverse eigenvalue problem}


% %    author one information
% % \author[short version for running head]{name for top of paper}
% \author{}
% \address{}
% \curraddr{}
% \email{}
% \thanks{}

% %    author two information
% \author{}
% \address{}
% \curraddr{}
% \email{}
% \thanks{}

%    \subjclass is required.

\date{}

\dedicatory{}

%    Abstract is required.
\begin{abstract}
The longstanding \emph{nonnegative inverse eigenvalue problem} (NIEP) is to determine which multisets of complex numbers occur as the spectrum of an entry-wise nonnegative matrix. Although there are some well-known necessary conditions, a solution to the NIEP is far from known. 

An invertible matrix is called a \emph{Perron similarity} if it diagonalizes an irreducible, nonnegative matrix. Johnson and Paparella [Linear Algebra Appl.,~{\bf 493} (2016), 281--300] developed the theory of real Perron similarities. Here, we fully develop the theory of complex Perron similarities. 

Each Perron similarity gives a nontrivial polyhedral cone and convex polytope of realizable spectra (thought of as vectors in complex Euclidean space). The extremals of these convex sets are finite in number, and their determination for each Perron similarity would solve the diagonalizable NIEP, a major portion of the entire problem. By considering Perron similarities of certain realizing matrices of Type I Karpelevi{\v{c}} arcs, large portions of realizable spectra are generated for a given positive integer. This is demonstrated by producing a nearly complete geometrical representation of the spectra of four-by-four stochastic matrices.

Similar to the Karpelevi{\v{c}} region, it is shown that the subset of complex Euclidean space comprising the spectra of stochastic matrices is compact and star-shaped. \emph{Extremal} elements of the set are defined and shown to be on the boundary.

It is shown that the polyhedral cone and convex polytope of the \emph{discrete Fourier transform (DFT) matrix} corresponds to the conical hull and convex hull of its rows, respectively. Similar results are established for multifold Kronecker products of DFT matrices and multifold Kronecker products of DFT matrices and Walsh matrices. These polytopes are of great significance with respect to the NIEP because they are \emph{extremal} in the region comprising the spectra of stochastic matrices.  

Implications for further inquiry are also given. 
\end{abstract}

\maketitle

%---------------------
\section{Introduction}
%---------------------

The \emph{nonnegative inverse eigenvalue problem (NIEP)} asks, for each positive integer $n$, which multi-sets $\Lambda = \{\lambda_1,\ldots, \lambda_n \}$ of $n$ complex numbers occur as the eigenvalues of an $n$-by-$n$ entry-wise nonnegative matrix? This has proven one of the most challenging problems in mathematics, and is certainly the most sought-after question in matrix analysis. Thus, a variety of sub-questions have been worthy goals. 

If $A \geqslant 0$ (entry-wise) is $n$-by-$n$ and with spectrum $\Lambda$, then $\Lambda$ is called \emph{realizable}, and $A$ is called a \emph{realizing matrix}. If the realizing matrix is required to be diagonalizable, then the resulting subproblem is called the \emph{diagonalizable NIEP} or the \emph{DNIEP}. There are differences between the two problems \cite[p.~214]{jmpp2018} and both are unsolved when $n > 4$. A solution to either has appeared far off. For additional information about the NIEP, and its numerous variants, there is a recent survey \cite{jmpp2018}.

It is known that if $\Lambda = \{ \lambda_1, \dots, \lambda_n \}$ is realizable and $A$ is a realizing matrix for $\Lambda$, then  
\begin{equation}
\rho \left(\Lambda\right) \coloneqq \max_{1 \leqslant k \leqslant n} \left\{ \vert\lambda_k\vert \right\} \in \Lambda,      \label{sprad}		
\end{equation}
\begin{equation}
\Lambda = \overline{\Lambda} \coloneqq \left\{ \overline{\lambda_1},\dots, \overline{\lambda_n} \right\},       \label{selfconj} 	
\end{equation}
\begin{equation} 
s_k (\Lambda) \coloneqq \sum_{i=1}^n \lambda_i^k = \trace{(A^k)} \geqslant 0,~\forall k \in \mathbb{N},               \label{trnn}		
\end{equation}
and
\begin{equation}
\left[ s_k (\Lambda) \right]^\ell \leqslant n^{\ell-1} s_{k\ell} (\Lambda), \forall k, \ell \in \mathbb{N}.           \label{JLL}  
\end{equation}
These conditions are not independent: Loewy and London \cite{ll1978-79} showed that the \emph{moment condition} \eqref{trnn} implies the \emph{self-conjugacy condition} \eqref{selfconj}. Friedland \cite[Theorem 1]{f1978} showed that the \emph{eventual nonnegativity} of the moments implies the spectral radius condition \eqref{sprad}. Finally, if the \emph{trace} is nonnegative, i.e., if \(s_1(\Lambda) \geqslant 0\), then the \emph{JLL condition} \eqref{JLL} (established independently by Johnson \cite{j1981} and by Loewy and London \cite{ll1978-79}) implies the moment condition since 
\[ s_\ell(\Lambda) \geqslant \frac{1}{n^{\ell-1}} \left[ s_1(\Lambda) \right]^\ell \geqslant 0, \forall \ell \in \mathbb{N}. \]
Thus, the JLL condition and the nonnegativity of the trace imply the self-conjugacy, spectral radius, and moment conditions. 

Holtz \cite{h2005} showed that if $\Lambda = \{ \lambda_1, \dots, \lambda_n \}$ is realizable, with $\lambda_1 = \rho\left(\Lambda\right)$, then the shifted spectrum $\{0,\lambda_1 - \lambda_2, \dots, \lambda_1- \lambda_n \}$ satisfies \emph{Newton's inequalities}. Furthermore, Holtz demonstrated that these inequalities are  independent of \eqref{trnn} an \eqref{JLL}.

The problem of characterizing the nonzero spectra of nonnegative matrices is due to Boyle and Handelman \cite{bh1991} (a constructive version of their main result was given by Laffey \cite{l2012}). However, despite their remarkable achievement, and the stringent necessary conditions listed above, the NIEP and its variants are unsolved when $n$ is greater than four.

Our focus here is upon the DNIEP (without loss of generality, when the realizing matrix is irreducible). We are able to explicitly characterize invertible matrices that diagonalize irreducible nonnegative matrices. We call them \emph{Perron similarities}. We show also that, for each Perron similarity, there is a nontrivial polyhedral cone of realizable spectra, which we call the \emph{(Perron) spectracone}. When the Perron similarity is properly normalized, the cross section of the spectracone, for which the spectral radius is one, is called the \emph{(Perron) spectratope} and is a convex polytope. With the mentioned normalization, each matrix may be taken to be row stochastic. This focuses attention, for each Perron similarity, upon the extreme points that are finite in number. Their determination solves a dramatic portion of the NIEP. This program is carried out for particular types of matrices, such as those that correspond to Type I Karpelevi{\v{c}} arcs (see below) and circulant and block circulant matrices.

%-------------------------------
\section{Notation \& Background}
%-------------------------------

For ease of notation, $\mathbb{N}$ denotes the set of positive integers and $\mathbb{N}_0 \coloneqq \mathbb{N} \cup \{ 0 \}$. If $n \in \mathbb{N}$, then $\bracket{n} \coloneqq \{ k \in \mathbb{N} \mid 1 \leqslant k \leqslant n \}$.

The set of $m$-by-$n$ matrices with entries from a field $\mathbb{F}$ is denoted by $\mathsf{M}_{m \times n}(\mathbb{F})$. If $m = n$, then $\mathsf{M}_{n \times n}(\mathbb{F})$ is abbreviated to $\mathsf{M}_n(\mathbb{F})$. 
% The set of all $n$-by-$1$ column vectors is identified with the set of all ordered $n$-tuples with entries in $\mathbb{F}$ and thus denoted by $\mathbb{F}^n$. 
The set of nonsingular matrices in $\mathsf{M}_n(\mathbb{F})$ is denoted by $\mathsf{GL}_{n}\left(\mathbb{F}\right)$.   

If $x \in \mathbb{F}^n$, then $x_k$ or $[x]_k$ denotes the $k\textsuperscript{th}$-entry of $x$ and $D_x = D_{x^\top} \in \mathsf{M}_n$ denotes the diagonal matrix whose $(i,i)$-entry is $x_i$. Notice that 
\[ D_{\alpha x + \beta y} = \alpha D_x + \beta D_y,\ \forall \alpha, \beta \in \mathbb{F}, \forall x, y \in \mathbb{F}^n. \] 

Denote by $I$, $e$, $e_i$, and $0$ the identity matrix, the all-ones vector, the $i\textsuperscript{th}$ canonical basis vector, and the zero vector, respectively. The size of each aforementioned object is implied by its context.

If $A \in \mathsf{M}_{m \times n}(\mathbb{F})$, then:
\begin{itemize}
    \item $a_{ij}$, $a_{i,j}$, or $[A]_{ij}$ denotes the $(i,j)$-entry of $A$;
    \item $A^\top$ denotes the \emph{transpose of $A$};
    \item $\overline{A} = [\overline{a_{ij}}]$ denotes the entrywise conjugate of $A$; 
    \item $A^\ast \coloneqq  \overline{A^\top} = \overline{A}^\top$ denotes the conjugate-transpose of $A$; and 
    \item $r_i(A) \coloneqq A^\top e_i$ denotes the $i\textsuperscript{th}$-row of $A$ as a column vector (when the context is clear, $r_i(A)$ is abbreviated to $r_i$).    
\end{itemize}
If $A \in \mathsf{M}_{n}(\mathbb{F})$, then $\spec A = \spec(A)$ denotes the \emph{spectrum of $A$} and $\rho = \rho(A)$ denotes the \emph{spectral radius of $A$}.

If $A \in \mathsf{M}_{n}(\mathbb{F})$ and $n \geqslant 2$, then $A$ is called \emph{reducible} if there is a permutation matrix $P$ such that
\begin{align*}
P^\top A P =
\begin{bmatrix}
A_{11} & A_{12} \\
0 & A_{22}
\end{bmatrix},
\end{align*}
where $A_{11}$ and $A_{22}$ are nonempty square matrices. If $A$ is not reducible, then A is called \emph{irreducible}. 

If $A \in \mathsf{M}_n(\mathbb{F})$, then the \emph{characteristic polynomial of $A$}, denoted by $\chi_A$, is defined by $\chi_A(t) = \det(tI - A)$. The \emph{companion matrix} $C = C_p$ of a monic polynomial $p(t) = t^n + \sum_{k=1}^{n} c_{k} t^{n - k}$ is the $n$-by-$n$ matrix defined by
\[ C = 
\left[\begin{array}{cc}
0 & I_{n-1} \\
-c_n & -c
\end{array} \right], \]
where $c = [c_{n-1}~\cdots~c_1]$. It is well-known that $\chi_{C_p} = p$. Notice that $C$ is irreducible if and only if $c_n \neq 0$.

% If $x \in \mathbb{C}^n$, then 
% $$\Re x \coloneqq 
% \begin{bmatrix}
%     \Re x_1 \\
%     \vdots \\
%     \Re x_n
% \end{bmatrix} \in \mathbb{R}^n$$
% and 
% $$\Im x \coloneqq 
% \begin{bmatrix}
%     \Im x_1 \\
%     \vdots \\
%     \Im x_n
% \end{bmatrix} \in \mathbb{R}^n.$$ 
% Since $\ii x = -\Im x + \ii \Re x$ and $-\ii x = \Im x - \ii \Re x$, it follows that 
%     \begin{equation}
%         \label{realandimag}
%         \Im x = 0 \iff (\Re(\ii x) 	\geqslant 0) \wedge (\Re(-\ii x) \geqslant 0).    
%     \end{equation}
% Similarly, if $A \in \mathsf{M}_n(\mathbb{C})$, then $\Re A \coloneqq \begin{bmatrix} \Re a_{ij} \end{bmatrix} \in \mathsf{M}_n(\mathbb{R})$ and $\Im A \coloneqq \begin{bmatrix} \Im a_{ij} \end{bmatrix} \in \mathsf{M}_n(\mathbb{R})$.  

If $x,y \in \mathbb{C}^n$, then $\langle x,y \rangle$ denotes the canonical inner product of $x$ and $y$, i.e., 
$$ \langle x,y \rangle = y^\ast x = \sum_{k=1}^n \overline{y_k} \cdot x_k $$ 
and
\[ \begin{Vmatrix} x \end{Vmatrix}_2 \coloneqq \sqrt{\langle x, x \rangle}. \]
% If $x,y \in \mathbb{C}^n$, then 
% \begin{equation}
%     \Re \left( \langle x,y \rangle \right) = \left(\Re x \right)^\top \Re y + \left(\Im x \right)^\top \Im y 
% \end{equation}

If $n \in \mathbb{N}$, then 
$$S^n \coloneqq \left\{ x \in \mathbb{C}^n \mid || x ||_\infty \coloneqq \max_{1 \leqslant k \leqslant n} \{|x_k| \} = 1 \right\}$$ 
and $B^n \coloneqq \{ x \in \mathbb{C}^n \mid  || x ||_\infty \leqslant 1 \}$. 

% For $n \in \mathbb{N}$, $n > 1$, we let $C_n$ denote the \emph{basic circulant}, i.e.,  
% \[ C_n =
% \left[ 
% \begin{array}{cc}
% 0 & I_{n-1} \\
% 1 & 0
% \end{array} \right]. \]
% Note that the digraph of $C_n$ is a cycle of length $n$.

The \emph{Hadamard product} of $A = [a_{ij}]$, $B = [b_{ij}] \in \mathsf{M}_{m \times n}(\mathbb{F})$, denoted by $A \circ B$, is the $m$-by-$n$ matrix whose $(i,j)$-entry is $a_{ij} b_{ij}$. If $x \in \mathbb{F}^n$ and $p\in \mathbb{N}$, then $x^p$ denotes the $p\textsuperscript{th}$-power of $x$ with respect to the Hadamard product, i.e., $[x^p]_k = x_k^p$. If $p=0$, then $x^p \coloneqq e$. If $x \in \mathbb{F}^n$ is \emph{totally nonzero}, i.e., $x_k \ne 0,\ \forall k \in \bracket{n}$, then $x^{-1}$ denotes the inverse of $x$ with respect to the Hadamard product, i.e., $[x^{-1}]_k = x_k^{-1}$. Notice that if $x$ is totally nonzero, then $(D_x)^{-1} = D_{x^{-1}}$.

The \emph{direct sum} of $A_1, \dots, A_\ell$, where $A_k \in \mathsf{M}_{n_k}(\mathbb{F})$, denoted by $A_1 \oplus \dots \oplus A_\ell$, or $\bigoplus_{k=1}^\ell A_k$, is the $n$-by-$n$ matrix 
\[ 
 \left[
 \begin{array}{ccc}
 A_1 &  & \multirow{2}{*}{\Large 0} \\
 \multirow{2}{*}{\Large 0} & \ddots &  \\
  &  & A_\ell
 \end{array}
 \right],\ n = \sum_{k=1}^\ell n_k. 
\]

If $\sigma \in \mathsf{Sym}(n)$ and $x \in \mathbb{C}^n$, then $\sigma(x)$ is the $n$-by-$1$ vector such that $[\sigma(x)]_k = x_{\sigma(k)}$ and $P_\sigma \in \mathsf{M}_n$ denotes the permutation matrix corresponding to $\sigma$, i.e., $P_\sigma$ is the the $n$-by-$n$ matrix whose $(i,j)$-entry is \(\delta_{\sigma(i),j}\), where $\delta_{ij}$ denotes the Kronecker delta. As is well-known, $\left( P_\sigma \right)^{-1} = P_{\sigma^{-1}} = \left(P_\sigma \right)^\top$. When the context is clear, $P_\sigma$ is abbreviated to $P$. Notice that $Px = \sigma(x)$.

If $k \in \bracket{n}$, then $P_k$ denotes the matrix obtained by deleting the $k\textsuperscript{th}$-row of $I_n$ and $\pi_k: \mathbb{F}^n \longrightarrow \mathbb{F}^{n-1}$ is the projection map defined by $\pi_k(x) = P_k x$.

%--------------------------------
\subsection{Nonnegative Matrices}

If $A \in \mathsf{M}_n(\mathbb{R})$ and $a_{ij} \geqslant 0,\ \forall i,j \in \bracket{n}$ or $a_{ij} > 0,\ \forall i,j \in \bracket{n}$, then $A$ is called \emph{nonnegative} or \emph{positive}, respectively, and we write $A \geqslant 0$ or $A > 0$, respectively. If $x \in \mathbb{C}^n$ ($A \in \mathsf{M}_n(\mathbb{C})$), then $x \geqslant 0$ (respectively, $A \geqslant 0$) if $\Re x \geqslant 0$ and $\Im x = 0$ (respectively, $\Re A \geqslant 0$ and $\Im A = 0$). 

If $A \geqslant 0$ and $$\sum_{j=1}^n a_{ij} = 1,~\forall~i \in \bracket{n},$$ then $A$ is called \emph{(row) stochastic}. If $A \geqslant 0$, then $A$ is stochastic if and only if $Ae = e$. Furthermore, if $A$ is stochastic, then $1 \in \spec(A)$ and $\rho(A) = 1$. It is known that the NIEP and the stochastic NIEP are equivalent (see, e.g., Johnson \cite[p.~114]{j1981}).

%------------------
\begin{observation}
    \thlabel{stochasticconvex}
    If $A_1,\ldots, A_m \in \mathsf{M}_n(\mathbb{R})$ are stochastic, then the matrix
    \[ A \coloneqq \sum_{k=1}^m \alpha_k A_k,\ \sum_{k=1}^m \alpha_k = 1,\ \alpha_k \geqslant 0,\ \forall k \in \bracket{m}, \]
    is stochastic.
\end{observation}

We recall the Perron--Frobenius theorem for irreducible matrices.

%------------------------------------------
\begin{theorem}
    [{\cite[Theorem 8.4.4]{hj2013}}] 
    \thlabel{pftirr}
        Let $A \in \mathsf{M}_{n}(\mathbb{R})$ be irreducible and nonnegative, and suppose that $n \geqslant 2$. Then 
        \begin{enumerate}
        [label=(\roman*)]
        \item $\rho(A) >0$;
        \item $\rho(A)$ is an algebraically simple eigenvalue of $A$;
        \item there is a unique positive vector $x$ such that $Ax = \rho(A) x$ and $e^\top x = 1$; and 
        \item there is a unique positive vector $y$ such that $y^\top A = \rho(A) y^\top$ and $e^\top y = 1$.
        \end{enumerate}
\end{theorem}

The vector $x$ in part (iii) of \thref{pftirr} is called the \emph{(right) Perron vector of $A$} and the vector $y$ in part (iv) of \thref{pftirr} is called the \emph{left Perron vector of A} \cite{hj2013}. 

%----------------------------
\section{Preliminary Results}

In this section, we cover more background and establish some ancillary results that will be useful in the sequel.

%--------------------------------------
\subsection{Complex Polyhedral Regions}

If $S \subseteq \mathbb{C}^n$, then the \emph{conical hull of $S$}, denoted by $\coni S = \coni(S)$, is defined by 
\[ \coni {S} = 
\begin{cases}
    \{ 0 \}, & S = \varnothing \\
    \left\{ \sum_{k=1}^m \alpha_k x_k \in \mathbb{C}^n \mid m \in \mathbb{N},~x_k \in S,~\alpha_k \geqslant 0 \right\}, & S \ne \varnothing.    
\end{cases} \]
i.e., when $S$ is nonempty, $\coni {S}$ consists of all \emph{conical combinations}. Similarly, the \emph{convex hull of $S$}, denoted by $\conv{S} = \conv(S)$ is defined by  
\[ \coni {S} = 
\begin{cases}
    \{ 0 \}, & S = \varnothing \\
    \left\{ \sum_{k=1}^m \alpha_k x_k \in \mathbb{C}^n \mid m \in \mathbb{N},~x_k \in S,\ \sum_{k=1}^m \alpha_k = 1,\ \alpha_k \geqslant 0 \right\}, & S \ne \varnothing    
\end{cases}. \]
The conical hull (convex hull) of a finite list $\{x_1,\ldots,x_n\}$ is abbreviated to $\coni(x_1,\ldots,x_n)$ (respectively, $\conv(x_1,\ldots,x_n)$).

% The \emph{interior of the conical hull} consists of all possible positive linear combinations, i.e., 
% \[ \interior(\coni{U}) = \left\{ \sum_{i=1}^k \alpha_i u_i \in V \mid k \in \mathbb{N},~u_i \in U,~\alpha_i > 0 \right\}. \]
%  is the set of all possible \emph{convex combinations} of vectors in $U$, i.e., 
% \[ \conv{U} = \left\{ \sum_{i=1}^k \alpha_i u_i \in V \mid k \in \mathbb{N},~u_i \in U,~\alpha \in \Delta^{k-1} \right\}.\]
% The \emph{interior of the convex hull} consists of all positive convex combinations, i.e., 
% \[ \interior(\conv{U}) = \left\{ \sum_{i=1}^k \alpha_i u_i \in V \mid k \in \mathbb{N},~u_i \in U,~\alpha \in \Delta^{k-1},\ \alpha > 0 \right\}.\]
A subset set $K$ of $\mathbb{C}^n$ is called
\begin{itemize}
    \item a \emph{cone} if $\coni x \subseteq K,\ \forall x \in K$; 
    \item \emph{convex} if $\conv(x,y) \subseteq K,\ \forall x,y \in K$;
    \item \emph{star-shaped} at $c \in K$ if $\conv(c,x) \subseteq K,\ \forall x \in K$; and 
    \item a \emph{convex cone} if $\coni(x,y) \subseteq K,\ \forall x,y \in K$. 
\end{itemize}

If $S,\ T \subseteq \mathbb{C}^n$, then $S+T$ denotes the \emph{Minkowski sum of $S$ and $T$}, i.e., $S + T \coloneqq \{ x \in \mathbb{C}^n \mid x = s + t, \ s \in S,\ t \in T\}$ and $-S \coloneqq \{ x \in \mathbb{C}^n \mid x = -s,\ s \in S \}$. If $K$ is a convex cone, then the \emph{dimension of $K$} is the quantity $\dim(K + (-K))$. If $K = \coni(x_1,\ldots,x_k)$, then $\dim K = \dim(\linearspan (x_1,\ldots,x_k))$.

A point $x$ of a convex cone $K$ is called an \emph{extreme direction} or \emph{extreme ray} if $u \notin K$ or $v \notin K$, whenever $x = \alpha u + (1-\alpha) v$, with \(\alpha \in (0,1)\) and $u,v$ linearly independent. A point $x$ of a convex set $K$ is called an \emph{extreme point} (of $K$) if $u \notin K$ or $v \notin K$, whenever \(x = \alpha u + (1-\alpha) v\), with \(\alpha \in (0,1)\) and $u \ne v$, i.e., $x$ does not lie in any open line segment contained in $K$. 

If $a \in \mathbb{C}^n$, $a \ne 0$, and $b \in \mathbb{R}$, then 
\[ \mathsf{H}(a,b) \coloneqq \left\{ x \in \mathbb{C}^n \mid \Re (\langle a, x \rangle) \geqslant b \right\} = \left\{ x \in \mathbb{C}^n \mid \Re (\langle x, a \rangle) \geqslant b \right\} \]
is a closed \emph{half-space} determined by the \emph{hyperplane} $\left\{ x \in \mathbb{C}^n \mid \Re (\langle x, a \rangle) = b \right\}$. 
% Since $\langle x,y \rangle = \overline{\langle y, x \rangle}$, it follows that
% \[ \Re (\langle a, x \rangle) \geqslant b \iff \Re (\langle x, a \rangle) \geqslant b, \]
If $b=0$, then the half-space $\mathsf{H}(a,b)$ is abbreviated to $\mathsf{H}(a)$ and contains the origin on its boundary.   

Any set of the form 
\[ \bigcap_{k=1}^m \mathsf{H}(a_k,b_k) = \bigcap_{k=1}^m \left\{ x \in \mathbb{C}^n \mid \Re (\langle a_k, x \rangle) \geqslant b_k \right\} \]
is called a \emph{polyhedron} and a bounded polyhedron is called a \emph{polytope}. Any set of the form 
\[ \bigcap_{k=1}^m \mathsf{H}(a_k) = \left\{ x \in \mathbb{C}^n \mid \Re (\langle a_k, x \rangle) \geqslant 0 \right\}\]
is called a \emph{polyhedral cone}.

Since $\langle\langle x,y \rangle\rangle \coloneqq \Re(\langle x, a \rangle)$ is a real inner product and $\mathbb{C}^n$ is a $2n$-dimensional real vector space, it follows that $\mathbb{C}^n$ is a $2n$-dimensional Euclidean space and the following celebrated result is applicable (see, e.g.,  \cite[Corollaries 2.13 and 2.14]{k1959} and \cite[pp.~170--178]{r1997}; or references in \cite[Remark 3.3]{b-i1969} or \cite[p.~87]{s1986}).

%--------------
\begin{theorem}
    [Farkas--Minkowski--Weyl]
        If $P$ is a polytope or a polyhedral cone in a Euclidean space $\mathbb{E}^n$, then there are vectors $x_1,\ldots, x_k$ such that 
        $$P = \conv(x_1,\ldots,x_k)$$ 
        or 
        $$P = \coni(x_1,\ldots,x_k),$$ 
        respectively.      
\end{theorem}

The following four propositions will be useful in the sequel; because they are easy to establish, their proofs are omitted.

%------------------
\begin{proposition}
    \thlabel{exptball}
        If $n \in \mathbb{N}$, then $e$ is an extreme point of the convex set $B^n$.    
\end{proposition}

% \begin{proof}
% For contradiction, suppose that there are vectors $x, y \in B^n$ such that $e = \alpha x + (1 - \alpha) y$ with $\alpha \in (0,1)$ and $x, y \neq e$. Since $x$ and $y$ are distinct from $e$ and belong to $B^n$, it follows that $\Re(x) < e$ and $\Re(y) < e$. Thus,  
% \begin{equation*}
% e = \Re(e) = \Re(\alpha x + \beta y) = \alpha \Re(x) + \beta \Re(y) < \alpha e + \beta e = e, 
% \end{equation*}
% a contradiction.
% \end{proof}

If $n \in \mathbb{N}$, then the \emph{standard (or unit) $n$-simplex}, denoted by $\Delta^n$, is defined by 
\begin{equation*}
\Delta^n = \left\{ (\alpha_1, \dots, \alpha_{n+1}) \in \mathbb{R}^{n+1} \mid \sum_{i=1}^{n+1} \alpha_i =1,~\alpha_i \geqslant 0 \right\}.
\end{equation*}

%------------------
\begin{proposition}
    \thlabel{exptconv}
        Let $K \subseteq \mathbb{C}^n$ be convex and suppose that $x$ is an extreme point of $K$. If there are vectors $x_1, \dots, x_m \in K$ and a vector $\alpha =[\alpha_1~\cdots~\alpha_m]^\top \in \Delta^{m-1}$ such that $x = \sum_{k=1}^m \alpha_k x_k$, then  $\alpha \in \{ e_1, \dots, e_m \}$.
\end{proposition}

% \begin{proof}
% For contradiction, if 
% \[ x = \sum_{k=1}^m \alpha_k x_k,~\alpha \in \Delta^{m-1} \backslash \{ e_1, \dots, e_m \}, \] 
% then $\alpha$ must have at least two nonzero entries. Without loss of generality, assume that $\alpha_1 \ne 0$. Notice that $1-\alpha_1 \ne 0$ and 
%     \begin{equation*}
%     x = \sum_{k=1}^m \alpha_k x_k = \alpha_1 x_1 + (1 - \alpha_1) \sum_{k=2}^m \frac{\alpha_k}{1 - \alpha_1} x_k.
%     \end{equation*}
% Because  
%     \begin{equation*}
%     \sum_{k=2}^m \frac{\alpha_k}{1 - \alpha_1} = \frac{1}{1 - \alpha_1} \sum_{k=2}^m \alpha_k = \frac{1}{1 - \alpha_1}(1 - \alpha_1) = 1
%     \end{equation*}
% and $K$ is convex, we have 
%     \begin{equation*}
%     \sum_{k=2}^m \frac{\alpha_k}{1 - \alpha_1} x_k \in K,
%     \end{equation*}
% Thus, $x$ lies on an open line segment contained in $K$, a contradiction.
% \end{proof}

%------------------
\begin{proposition}
    \thlabel{cor:allonesextreme}
        If $S = \{x_1,\dots, x_k\} \subseteq B^n$, then $e \in \conv{(S)}$ if and only if $e \in S$.
\end{proposition}

% \begin{proof}
% The sufficiency of the condition is trivial and the necessity of the condition follows from Propositions \ref{exptball} and \ref{exptconv} and the fact that $B^n$ is convex.
% \end{proof}

%------------------
\begin{proposition}
    \thlabel{conisub}
    If $S$ is a finite subset of a convex cone $K$, then $\coni(S) \subseteq K$. 
\end{proposition}

% \begin{proof}
%     The result is clear when $S$ is empty; if $S = \{x_1\} \subset K$ and $v \in \coni(S)$, then 
%     \[ v = \alpha x_1 = \frac{\alpha}{2}x_1 + \frac{\alpha}{2}x_1 \in K. \]

%     For the induction-step, assume that the result holds for all sets containing $k \geqslant 1$ vectors. If $S = \{ x_1, \ldots, x_{k+1}\} \subset K$ and $v \in \coni(S)$, then
%     \begin{align*}
%         v 
%         = \sum_{j=1}^{k+1} \alpha_j x_j 
%         &= \left( \sum_{j=1}^{k} \alpha_j x_j \right) + \alpha_{k+1} x_{k+1}        \\ 
%         &= 1\left( \sum_{j=1}^{k} \alpha_j x_j \right) + \alpha_{k+1} x_{k+1} \in K.
%     \end{align*}
%     The entire result follows by the principle of mathematical induction.
% \end{proof}

%------------------
\begin{proposition}
    \thlabel{conecontain}
        Let $K$ and $\hat{K}$ be convex cones. If $K = \coni(x_1,\dots,x_m)$, then $K \subseteq \hat{K}$ if and only if $x_1,\dots,x_m \in \hat{K}$.
\end{proposition}

% \begin{proof}
%     Necessity is clear since $x_i \in K$ and sufficiency is an immediate consequence of \thref{conisub}.  
% \end{proof}

%----------------------------------------------
\subsection{Region Comprising Stochastic Lists}

For $x \in \mathbb{C}^n$, denote by $\Lambda(x)$ the list $\{ x_1, \dots, x_n\}$ and for every natural number $n$, let 
\[ \mathbb{L}^n \coloneqq \{ x \in \mathbb{C}^n \mid \Lambda(x) = \spec(A),~A \in \mathsf{M}_{n}\left( \mathbb{R} \right),~A\geqslant 0 \} \]
and 
\[ \mathbb{SL}^n \coloneqq \{ x \in \mathbb{C}^n \mid \Lambda(x) = \spec(A),~A \in \mathsf{M}_{n}\left( \mathbb{R} \right),~A~\text{stochastic} \}. \]
Clearly, $\mathbb{L}^n$ is a cone that contains $\mathbb{SL}^n$ and a characterization of either set constitutes a solution to the NIEP. As such, we catalog various properties of each set. 

Recall that if $x,y \in \mathbb{C}^n$, then the angle $\theta = \theta(x,y)$ between them is defined by
\begin{equation}
    \label{complexangle}
    \theta = \theta(x,y) =
    \begin{cases}
        \frac{\pi}{2}, & (x = 0) \vee (y = 0) \\
        \arccos \frac{\Re (\langle x, y \rangle)}{\vert\vert x \vert\vert_2 \cdot \vert\vert y \vert\vert_2}, & (x \ne 0) \wedge (y \ne 0) 
    \end{cases}.
\end{equation}
Because non-real eigenvalues of real matrices occur in complex conjugate pairs, it follows that $x \in \mathbb{L}^n \iff \bar{x} \in \mathbb{L}^n$. Thus, $\Lambda(x) = \overline{\Lambda(x)} = \Lambda(\bar{x})$ whenever $x\in \mathbb{L}^n$ or $\bar{x}\in \mathbb{L}^n$. Furthermore, if $x \in \mathbb{L}^n$, then there is a nonnegative matrix $A$ such that $\spec A = \Lambda(x) = \Lambda(\bar{x})$. Consequently, 
\begin{equation}
    \label{anglexe}
    \langle x, e \rangle = e^\ast x = e^\top x = \sum_{k=1}^n x_k = \trace{A} \geqslant 0
\end{equation}
and 
\begin{equation}
    \label{anglexbare}
    \langle \overline{x}, e \rangle = e^\top \bar{x} = \sum_{k=1}^n \overline{x_k} = \trace{A} \geqslant 0
\end{equation}
since the realizing matrix is nonnegative.

%------------------
\begin{proposition}
    \thlabel{SLangle}
    If $x \in \mathbb{L}^n$, then $\theta(x,e) \in [0,\pi/2]$ and $\theta(\bar{x},e) \in [0,\pi/2]$. 
\end{proposition}

\begin{proof}
    Immediate from \eqref{complexangle}, \eqref{anglexe}, and \eqref{anglexbare}.
\end{proof}

%------------
\begin{lemma}
    \thlabel{seq}
    If $\{ A_k \}_{k=1}^\infty$ is a convergent sequence of stochastic $n$-by-$n$ matrices with limit $L$, then $L$ is stochastic.  
\end{lemma}

\begin{proof}
    Routine analysis exercise.
\end{proof}

% \begin{proof}
%     Denote by $a_{ij}^k$ the $(i,j)$-entry of $A_k$. Since every matrix in the sequence is stochastic, it follows that 
%     $$\sum_{j=1}^n \ell_{ij} = \sum_{j=1}^n \lim_{k \rightarrow \infty} \left( a_{ij}^k \right) = \lim_{k \rightarrow \infty} \left( \sum_{j=1}^n a_{ij}^k \right) = \lim_{k \rightarrow \infty} 1 = 1,\ \forall i \in \bracket{n},$$ 
%     i.e., every row of $L$ sums to unity. 
    
%     Next, we demonstrate that $L \geqslant 0$.  To the contrary, suppose that $L$ contains a negative entry, say $\ell_{pq}$. Since $A_k \longrightarrow L$, for every $\varepsilon > 0$, there is an integer $K = K(\varepsilon)$ such that $\begin{Vmatrix} A_k - L \end{Vmatrix}_\infty < \varepsilon,\ \forall k \geqslant K$. In particular, if $\varepsilon > 0$ and $\ell_{pq} + \varepsilon < 0$, then 
%     \[ \vert a_{pq}^k - \ell_{pq} \vert \leqslant \sum_{j=1}^n \vert a_{pj}^k - \ell_{pj} \vert \leqslant \begin{Vmatrix} A_k - L \end{Vmatrix}_\infty < \varepsilon \]
%     i.e., $\ell_{pq} - \varepsilon < a_{pq}^k < \ell_{pq} + \varepsilon < 0$, a contradiction.   
% \end{proof}

%--------------
\begin{theorem}
    If $n$ is a positive integer, then $\mathbb{SL}^n$ is compact. 
\end{theorem}

\begin{proof}
Following \thref{seq} and the fact that the eigenvalues of a matrix are topologically continuous \cite[pp.~620--621]{lz2019}, it follows that $\mathbb{SL}^n$ is closed. Because the spectral radius of a stochastic matrix is one, we obtain $\mathbb{SL}^n \subseteq S^n$, i.e., $\mathbb{SL}^n$ is bounded. 
\end{proof}

%------------
\begin{theorem}
    If $n$ is a positive integer, then $\mathbb{SL}^n$ is star-shaped at $e$.    
\end{theorem}

\begin{proof}
If $x \in \mathbb{SL}^n$, then there is a stochastic matrix $A$ such that $\Lambda(x) = \spec(A)$. Write $\spec(A)=\{1,x_2,\ldots,x_n\}$ and let $S$ be an invertible matrix such that $J = S^{-1} A S$ is a Jordan canonical form of $A$. If $\alpha \in [0,1]$ and $\beta \coloneqq  1 - \alpha$, then the matrix $\alpha A + \beta I$ is stochastic (\thref{stochasticconvex}) and since $S^{-1}\left(\alpha A + \beta I\right)S = \alpha J + \beta I$ is upper-triangular, it follows that $\spec(\alpha A + \beta I) = \{1,\alpha x_2 + \beta,\ldots,\alpha x_n + \beta\}$, i.e., $\alpha x + \beta e \in \mathbb{SL}^n$, $\forall \alpha \in [0,1]$.
\end{proof}

If $x \in \mathbb{SL}^n$, then $1 = \rho\left(\Lambda(x)\right) \in \Lambda(x)$ and if $P$ is a permutation matrix corresponding to the permutation $\sigma$, then \( \Lambda(Px) = \Lambda(\sigma(x)) = \Lambda(x) \). In light of these two facts, there is no loss in generality in restricting attention to $\mathbb{SL}_1^n \coloneqq \left\{ x \in \mathbb{SL}^n \mid x_1 = 1 \right\}$. 

The following eigenvalue-perturbation result is due to Brauer \cite[Theorem 27]{b1952} (for more proofs, see \cite{mp2021} and references therein). 

%--------------
\begin{theorem}
    [Brauer]
    \thlabel{Brauer}
        Let $A \in \mathsf{M}_n(\mathbb{C})$ and suppose that 
        $$\spec(A) = \{ \lambda_1, \ldots, \lambda_k, \dots, \lambda_n \}.$$ 
        If $x$ is an eigenvector associated with $\lambda_k$ and $y \in \mathbb{C}^n$, then $\spec(A + xy^*) = \{ \lambda_1, \ldots, \lambda_k + y^*x, \ldots, \lambda_n \}$. 
\end{theorem}

%--------------
\begin{theorem}
    If $n > 1$, then $\pi_1(\mathbb{SL}_1^n)$ is star-shaped at the origin. 
\end{theorem}

\begin{proof}
If $x \in \mathbb{SL}_n$, then there is a stochastic matrix $A$ such that $\Lambda(x) = \spec A$. If $\spec A = \{1,x_2,\ldots,x_n\}$ and $\alpha \in \mathbb{C}$, then $\spec(\alpha A) = \{ \alpha, \alpha x_2, \ldots, \alpha x_n \}$. By \thref{stochasticconvex}, the matrix 
$$\alpha A + \frac{1-\alpha}{n}ee^\top$$ 
is stochastic. Since $(\alpha A)e = \alpha e$, it follows that  
\begin{align*}
\spec \left( \alpha A + \frac{1-\alpha}{n}ee^\top \right) 
&= \spec \left( \alpha A + e \left( \frac{1-\alpha}{n} e^\top\right) \right)                                        \\
&= \left\{ \alpha + \frac{1-\alpha}{n} (e^\top e), \alpha x_2,\ldots,\alpha x_n \right\}  \tag{\thref{Brauer}}   \\
&= \{ \alpha + (1 - \alpha), \alpha x_2,\ldots, \alpha x_n \}                                                       \\
&=\{ 1, \alpha x_2, \ldots, \alpha x_n \}. 
\end{align*}
Thus, $\alpha \pi_1(x) \in \pi_1(\mathbb{SL}^n), \forall\alpha\in[0,1]$. 
\end{proof}

% %--------------
% \begin{theorem}
%     \thlabel{cyclicpoly}
%     If $x \in \mathbb{SL}_1^n$ and $p \in \mathbb{N}_0$, then $\conv(e, \ldots, x^p) \subseteq \mathbb{SL}_1^n$.
% \end{theorem}

% \begin{proof}
%     By hypothesis, there is a stochastic matrix $A$ such that $\spec A = \Lambda(x)$. If $y \in \conv(e, \ldots, x^p)$, then  
%     \[ y = \sum_{k=0}^p \alpha_k x^k,\ \sum_{k=0}^p \alpha_k = 1,\ \alpha_k \geqslant 0,\ 0 \leqslant k \leqslant p. \]
%     If $k \in \mathbb{N}_0$, then $A^k$ is a stochastic matrix with spectrum $\Lambda(x^k)$. Thus, by \thref{stochasticconvex}, the matrix $M = \sum_{k=0}^p \alpha_k A^p$ is a stochastic matrix with spectrum $\Lambda(y)$. 
% \end{proof}

%---------------------------------------
\subsection{The Karpelevi{\v{c}} Region}

In 1938, Kolmogorov posed the question of characterizing the region in the complex plane, denoted by $\Theta_n$, that occur as an eigenvalue of a stochastic matrix \cite[p.~2]{s1972}. Dmitriev and Dynkin \cite{dd1946} (see \cite[Appendix A]{s1972} or \cite{ams140} for an English translation) obtained a partial solution, and Karpelevi{\v{c}} \cite{k1951} (see \cite{ams140} for an English translation) solved the problem by showing that the boundary of $\Theta_n$ consists of curvilinear \emph{arcs} (hereinafter, \emph{Karpelevi{\v{c}} arcs} or \emph{K-arcs}), whose points satisfy a polynomial equation that depends on the endpoints of the arc. 

If $n \in \mathbb{N}$, then $\mathcal{F}_n \coloneqq  \{ p/q \mid 0\leq p \leqslant q \leq n,~\gcd(p,q)=1 \}$ denotes the set of \emph{Farey fractions}. The following result is the celebrated Karpelevi{\v{c}} theorem in a form due to Ito \cite{i1997}. 

%--------------
\begin{theorem}
    [Karpelevi{\v{c}} theorem] 
        \thlabel{karpito}
    The region $\Theta_n$ is symmetric with respect to the real axis, is included in the unit-disc $\{ z \in \mathbb{C} \mid |z| \leq 1\}$, and intersects the unit-circle $\{ z \in \mathbb{C} \mid |z| = 1\}$ at the points $\left\{ e^{\frac{2\pi p}{q}\ii} \mid p/q \in \mathcal{F}_n \right\}$. The boundary of $\Theta_n$ consists of these points and of curvilinear arcs connecting them in circular order. 
    
    Let the endpoints of an arc be $e^{\frac{2\pi p}{q}\ii}$ and $e^{\frac{2\pi r}{s}\ii}$ ($q \leqslant s$). Each of these arcs is given by the following parametric equation:  
        \begin{equation}
            \label{itoequation}
                t^{s} \left( t^{q} - \beta \right)^{\floor{n/q}} = \alpha^{\floor{n/q}} t^{q\floor{n/q}},\ \alpha \in [0,1], ~\beta\coloneqq 1-\alpha.
        \end{equation} 
\end{theorem}

Following \cite{jp2017}, equation \eqref{itoequation} is called the \emph{Ito equation} and the polynomial
    \begin{equation}
        \label{itopolynomial}
            f_\alpha(t)\coloneqq  t^s(t^q-\beta)^{\floor{n/q}}-\alpha^{\floor{n/q}}t^{q\floor{n/q}},\ \alpha\in[0,1],\ \beta \coloneqq 1-\alpha    
    \end{equation}
is called the \textit{Ito polynomial}. The Ito polynomials are divided into four types as follows (note that $s \ne q\floor{n/q}$ since $\gcd{(q,s)} = 1$): 
\begin{itemize}
    \item If $\floor{n/q} = n$, then  
        \begin{equation}
            \label{type_zero_ito}
                f_\alpha^{\mathsf{0}} (t) = (t-\beta)^n - \alpha^n,~\alpha\in[0,1],\ \beta \coloneqq 1-\alpha.
        \end{equation}
    is called a \emph{Type 0 (Ito) polynomial} and corresponds to the \emph{Farey pair} $(0/1,1/n)$.

    \item If $\floor{n/q} = 1$, then  
        \begin{equation}
            \label{type_one_ito}
                f_\alpha^{\mathsf{I}}(t) = t^s-\beta t^{s-q} - \alpha,\ \alpha\in[0,1],\ \beta \coloneqq 1-\alpha.
        \end{equation}
    is called a \emph{Type I (Ito) polynomial}.
    
    \item If $1 < \floor{n/q} < n$ and $s > q\floor{n/q}$, then 
        \begin{equation}
            \label{type_two_ito}
                f_\alpha^{\mathsf{II}} (t) = (t^q-\beta)^{\floor{n/q}}-\alpha^{\floor{n/q}}t^{q\floor{n/q}-s},\ \alpha\in[0,1],\ \beta \coloneqq 1-\alpha.
        \end{equation}
    is called a \emph{Type II (Ito) polynomial}.
    
    \item If $1 < \floor{n/q} < n$ and $s < q\floor{n/q}$, then 
        \begin{equation}
            \label{type_tre_ito}
                f_\alpha^{\mathsf{III}} (t) = t^{s-q\floor{n/q}}(t^q-\beta)^{\floor{n/q}}-\alpha^{\floor{n/q}},\ \alpha\in[0,1],\ \beta \coloneqq 1-\alpha. 
        \end{equation}
    is called a \emph{Type III (Ito) polynomial}.
\end{itemize}
The polynomials given by equations \eqref{type_zero_ito}--\eqref{type_tre_ito} are called the \emph{reduced Ito polynomials}. Johnson and Paparella \cite[Theorem 3.2]{jp2017} showed that if $\alpha \in [0,1]$, then there is a stochastic matrix $M = M(\alpha)$ such that $\chi_M = f_\alpha^\mathsf{X}$, where $\mathsf{X} \in \{ \mathsf{0},\mathsf{I},\mathsf{II}, \mathsf{III} \}$. 

%-------------
\begin{remark}
    \label{rem:distinctroots}
    If $p$ is a polynomial and $p'$ denotes its derivative, then $p$ has a multiple root if and only if the \textit{resultant} $R(p,p') \coloneqq \det(S(p,p')^\top)$ vanishes (here $S(p,p')$ denotes the \textit{Sylvester matrix}). Since the coefficients of the polynomials $f_\alpha$ and $f_\alpha'$ depend on the single parameter $\alpha$, it follows that $\pi(\alpha) \coloneqq \det(S(f_\alpha,f_\alpha')^\top)$ is a univariate polynomial in $\alpha$ of degree at most 
    \[\deg f_\alpha + (\deg f_\alpha-1) = 2\deg f_\alpha - 1. \] 
    Hence, there are at most $2\deg f_\alpha - 1$ zeros and at most $2\deg f_\alpha - 1$ values in $[0,1]$ such that $f_\alpha$ does not have distinct zeros.
\end{remark}

More is known about the number of zeros of $f_\alpha$ corresponding to the Type I arc $K_n({1}/{n}, {1}/{(n-1)})$.

%-----------
\begin{proposition}
    \label{distincteigs}
    {\cite[Proposition 4.1]{jp2016}}
    For $n \geqslant 4$, let 
    \begin{equation}
    f_\alpha (t) \coloneqq t^n - \beta t - \alpha, ~\alpha \in [0,1], ~\beta \coloneqq 1 - \alpha. \label{polyalpha}
    \end{equation} 
    \begin{enumerate}
    [label=(\roman*)]
    \item If $n$ is even, then $f_\alpha$ has $n$ distinct roots.
    \item If $n$ is odd and $\alpha \geqslant \beta$,  then $f_\alpha$ has $n$ distinct roots.
    \item If $n$ is odd and $\alpha < \beta$,  then $f_\alpha$ has a multiple root if and only if \( n^n \alpha^{n-1} - (n-1)^{n-1} \beta^n = 0 \).   
    \end{enumerate}
\end{proposition}

If $f_\alpha$ is defined as in \eqref{polyalpha}, $n$ is odd, $\alpha < \beta$, and \( \pi(\alpha) = (n-1)^{n-1} \beta^n - n^n \alpha^{n-1} \), then it is known that the polynomial $\pi$ has only one zero in the interval $[0,1]$ \cite[Remark 4.3]{jp2017}.

%----------------------------------------
\subsection{Extremal and Boundary Points}

If $\lambda \in \Theta_n$, then $\lambda$ is called \emph{extremal} if $\alpha \lambda \not \in \Theta_n$ whenever $\alpha > 1$. It is an easy exercise to show that if $\lambda$ is extremal, then $\lambda \in \partial \Theta_n$. Karpelevi{\v{c}} asserted that the converse follows from the closure of $\Theta_n$ but this is incorrect in view of the two- and three-dimensional cases. However, an elementary proof was given recently by Munger et al.~\cite[Section 6]{mnp2024}.

%-----------------
\begin{definition}
    If $x = \begin{bmatrix} 1 & x_2 & \cdots & x_n \end{bmatrix}^\top \in \mathbb{SL}_1^n$, then $x$ is called \emph{extremal (in $\mathbb{SL}_1^n$)} if $\alpha \pi_1(x) \not \in \pi_1(\mathbb{SL}_1^n)$, $\forall\alpha > 1$. The set of extremal points in $\mathbb{SL}_1^n$ is denoted by $\mathbb{E}_n$. 
\end{definition}

%--------------
\begin{theorem}
    If $n$ is a positive integer, then $\mathbb{E}_n \subseteq \partial \mathbb{SL}_1^n$.
\end{theorem}

\begin{proof}
    If $n=1$, then $\mathbb{SL}_1^n = \{ 1 \}$ and the result is clear. 
    
    Otherwise, assume that $n > 1$ and, for contradiction, that $x \in \mathbb{E}_n$, but $x \notin \partial \mathbb{SL}_1^n$. By definition, $\exists \varepsilon > 0$ such that $N_\varepsilon(x) \coloneqq \{ y \in \mathbb{C}^n \mid \begin{Vmatrix} y - x \end{Vmatrix}_\infty < \varepsilon \} \subseteq \mathbb{SL}_1^n$. If $\alpha \coloneqq 1 + \varepsilon$ and 
    $$y \coloneqq \begin{bmatrix} 1 \\ \pi_1(\alpha x) \end{bmatrix} = \begin{bmatrix} 1 \\ \alpha x_2 \\ \vdots \\ \alpha x_n \end{bmatrix},$$ 
    then 
    $$y - x 
    = \begin{bmatrix} 0 \\ (\alpha-1) x_2 \\ \vdots \\ (\alpha-1) x_n \end{bmatrix} 
    = \varepsilon \begin{bmatrix} 0 \\ x_2 \\ \vdots \\ x_n \end{bmatrix}.$$
    Because $\begin{Vmatrix} x \end{Vmatrix}_\infty = 1$, it follows that $\begin{Vmatrix} \pi_1(x) \end{Vmatrix}_\infty \leqslant 1$ and $\begin{Vmatrix} y - x \end{Vmatrix}_\infty = \varepsilon \begin{Vmatrix} \pi_1(x) \end{Vmatrix}_\infty \leqslant \varepsilon$, i.e., $y \in N_\varepsilon(x)$. Since $\alpha > 1$, it follows that $x$ is not extremal, a contradiction. 
\end{proof}

% %--------------
% \begin{theorem}
%     If $n > 3$, then $\partial \mathbb{SL}_1^n \subseteq \mathbb{E}_n$.
% \end{theorem}

% \begin{proof}
%     Romanovsky \cite{r1936} proved that $\Theta_n$ intersects the unit-circle $\{ z \in \mathbb{C} \mid |z| = 1\}$ at the points $\{ \omega_q^p \mid p/q \in \mathcal{F}_n \}$. 
%     For contradiction, if $x \in \partial \mathbb{SL}_1^n$ and $x \notin \mathbb{E}_n$, then $\exists \alpha > 1$ such that $\alpha x \notin \mathbb{SL}_1^n$. By \thref{cyclicpoly}, $\conv(e,\alpha x, \ldots, \alpha^p x^p) \subseteq \mathbb{SL}_1^n$ whenever $p \geqslant 0$. 
% \end{proof}

%------------------
\begin{observation}
    If $x \in \mathbb{SL}_1^n$ and $x_k$ is extremal in $\Theta_n$, where $1 < k \leqslant n$, then $x \in \mathbb{E}_n$.
\end{observation}

\begin{proof}
    For contradiction, if $x \notin \mathbb{E}_n$, then $\exists \alpha > 1$ such that $\pi_1(x) \in \pi_1(\mathbb{SL}_1^n)$. Thus, there is a stochastic matrix $A$ with spectrum $\{1, \alpha x_2, \ldots,\alpha_n x_n \}$. Consequently, $\alpha x_k \in \Theta_n$, a contradiction.
\end{proof}

%---------------------------
\section{Spectral Polyhedra}
%---------------------------

Although the following results are specified for complex matrices, many of the definitions and results apply, with minimal alteration, to real matrices. 

%-----------------
\begin{definition}
    \thlabel{spectratope}
    If $S \in \mathsf{GL}_{n}(\mathbb{C})$, then: 
    \begin{enumerate}
        [label=(\roman*)]
        \item $\mathcal{C}(S) \coloneqq \{ x \in \mathbb{C}^n \mid M_x \coloneqq S D_x S^{-1} \geqslant 0 \}$ is called the \emph{(Perron) spectracone} of $S$;
        % \item $\mathcal{C}_0(S) \coloneqq \{ x \in \mathcal{C}(S) \mid \langle x, e \rangle = e^\top x = 0 \}$ is called the \emph{trace-zero (Perron) spectracone} of $S$;;
        \item $\mathcal{P}(S) \coloneqq \{ x \in \mathcal{C}(S) \mid M_x e = e \}$ is called the \emph{(Perron) spectratope} of $S$; 
        % \item $\mathcal{P}_0(S) \coloneqq \{ x \in \mathcal{P}(S) \mid \langle x, e \rangle = e^\top x = 0 \}$ is called the \emph{trace-zero (Perron) spectratope} of $S$; and 
        \item $\mathcal{A}(S) \coloneqq \{ M_x \in \mathsf{M}_n(\mathbb{R}) \mid x \in \mathcal{C}(S) \}$. 
    \end{enumerate}
\end{definition}

%----------
\begin{remark}
    Although the spectratope definitions that appeared in the literature previously (\cite[Definition 3.5]{jp2016} and \cite[p.~114]{dppt2022}) differ from what appears in \thref{spectratope}, the definition above subsumes the previous definitions and captures the notion in its fullest generality. 
\end{remark}

%------------------
\begin{observation}
    \thlabel{prodstochequalsstoch}
    The product of stochastic matrices is stochastic. 
\end{observation}

% \begin{proof}
% If $A,B \geqslant 0$ and $Ae = e = Be$, then $AB \geqslant 0$ and $(AB)e = A(Be) = Ae = e$.    
% \end{proof}

%--------------
\begin{theorem}
    \thlabel{hadamardcones}
    If $S \in \mathsf{GL}_{n}(\mathbb{C})$, then: 
    \begin{enumerate}
        [label=(\roman*)]
        \item \label{CShadamard} $\mathcal{C}(S)$ is a nonempty convex cone that is closed with respect to the Hadamard product; 
        \item $\mathcal{P}(S)$ is a nonempty convex set that is closed with respect to the Hadamard product; and 
        \item $\mathcal{A}(S)$ is a nonempty convex cone that is closed with respect to matrix multiplication.
    \end{enumerate}
\end{theorem}

\begin{proof}
Since $M_e = S D_e S^{-1} = S I_n S^{-1} = I_n \geqslant 0$ and $I_n$ is stochastic, it follows that $e \in \mathcal{P}(S) \subset \mathcal{C}(S)$ and all three sets are nonempty. 

\begin{enumerate}
    [label=(\roman*)]
    \item If $x, y \in \mathcal{C}(S)$ and $\alpha$, $\beta \geqslant 0$, then 
    \begin{equation}
        \label{convexcombo}
        M_{\alpha x + \beta y} = SD_{\alpha x + \beta y} S^{-1} = S(\alpha D_x + \beta D_y)S^{-1} = \alpha M_x + \beta M_y \geqslant 0,
    \end{equation} 
    i.e., $\mathcal{C}(S)$ is a convex cone. 
    
    Furthermore, 
    \begin{equation}
        \label{hadamardprod}
        M_{x\circ y} = SD_{x\circ y} S^{-1} = S D_x D_y S^{-1} = (S D_x S^{-1})(S D_y S^{-1}) = M_x M_y \geqslant 0,
    \end{equation}
    i.e., the convex cone $\mathcal{C}(S)$ is closed with respect to the Hadamard product.

    \item The convexity of $\mathcal{P}(S)$ follows from \eqref{convexcombo} and \thref{stochasticconvex}; closure with respect to the Hadamard product is a consequence of \eqref{hadamardprod} and \thref{prodstochequalsstoch}.  

    \item Follows from \eqref{convexcombo} and \eqref{hadamardprod}. \qedhere
\end{enumerate}
\end{proof}

%-------------
\begin{remark}
    If $\mathcal{C}(S) = \coni(e)$, $\mathcal{P}(S) = \{e\}$, or $\mathcal{A}(S) = \coni(I_n)$, then $\mathcal{C}(S)$, $\mathcal{P}(S)$, and $\mathcal{A}(S)$ are called \emph{trivial}; otherwise, they are called \emph{nontrivial}.
\end{remark}

Before we prove our next result, we note the following: if $x \in \mathbb{C}^n$, then 
$$\Re x \coloneqq 
\begin{bmatrix}
    \Re x_1 \\
    \vdots \\
    \Re x_n
\end{bmatrix} \in \mathbb{R}^n$$
and 
$$\Im x \coloneqq 
\begin{bmatrix}
    \Im x_1 \\
    \vdots \\
    \Im x_n
\end{bmatrix} \in \mathbb{R}^n.$$ 
Since $\ii x = -\Im x + \ii \Re x$ and $-\ii x = \Im x - \ii \Re x$, it follows that 
    \begin{equation}
        \label{realandimag}
        \Im x = 0 \iff (\Re(\ii x) 	\geqslant 0) \wedge (\Re(-\ii x) \geqslant 0).    
    \end{equation}
Similarly, if $A \in \mathsf{M}_n(\mathbb{C})$, then $\Re A \coloneqq \begin{bmatrix} \Re a_{ij} \end{bmatrix} \in \mathsf{M}_n(\mathbb{R})$ and $\Im A \coloneqq \begin{bmatrix} \Im a_{ij} \end{bmatrix} \in \mathsf{M}_n(\mathbb{R})$. 

%--------------
\begin{theorem}
    \thlabel{CSpolyconePSpolytope}
        If $S \in \mathsf{GL}_n(\mathbb{C})$, then $\mathcal{C}(S)$ is a polyhedral cone and $\mathcal{P}(S)$ is a polytope. 
\end{theorem}

\begin{proof}
    In what follows, we let $t_{ij}$ denote the $(i,j)$-entry of $S^{-1}$. Via the mechanics of matrix multiplication and the complex inner product, 
    \begin{equation}
        \label{Mijentry}
        \left[ M_x \right]_{ij} 
        = \left[ SD_xS^{-1} \right]_{ij} 
        = \sum_{k=1}^n \left( s_{ik} t_{kj} \right) x_k 
        = \left( s_i \circ t_j \right)^\top x 
        = \langle x,\overline{s_i \circ t_j} \rangle,
    \end{equation}
    where $s_i$ denotes the $i\textsuperscript{th}$-row of $S$ (as a column vector) and $t_j$ denotes the $j\textsuperscript{th}$-column of $S^{-1}$. Consequently,
    \begin{align*}
        x \in \mathcal{C}(S) 
        &\iff (\Re M_x \geqslant 0) \wedge(\Im M_x = 0)                         \\ 
        &\iff   
        x \in \mathsf{H}\left(\overline{s_i \circ t_j}\right) \cap 
        \mathsf{H}\left(\ii \cdot \overline{s_i \circ t_j}\right) \cap  
        \mathsf{H}\left(-\ii \cdot\overline{s_i \circ t_j}\right), \ \forall (i,j) \in \bracket{n}^2. 
    \end{align*}

    Via the mechanics of matrix multiplication and the complex inner product, notice that
    \begin{align*} 
    \left[ M_x e \right]_i 
    = \sum_{j=1}^n \sum_{k=1}^n \left( s_{ik} t_{kj} \right) x_k      
    = \sum_{k=1}^n \left( s_{ik} \cdot \sum_{j=1}^n t_{kj} \right) x_k 
    = \left( s_i \circ Te \right)^\top x
    = \langle x, \overline{s_i \circ Te} \rangle. 
    \end{align*}
    Thus, $x \in \mathcal{P}(S)$ if and only if 
    \[ x \in \mathsf{H}\left(\overline{s_i \circ t_j}\right) \cap 
            \mathsf{H}\left(\ii \cdot \overline{s_i \circ t_j}\right) \cap  
            \mathsf{H}\left(-\ii \cdot\overline{s_i \circ t_j}\right), \ \forall (i,j) \in \bracket{n}^2 \]
    and 
    \[ x \in 
    \mathsf{H}\left( \overline{s_i \circ Te},1 \right) \cap 
    \mathsf{H}\left( \overline{s_i \circ Te},-1 \right) \cap
    \mathsf{H}\left(  \ii\cdot \overline{s_i \circ Te} \right) \cap
    \mathsf{H}\left(  -\ii\cdot \overline{s_i \circ Te} \right). \]
    Thus, $\mathcal{P}(S)$ is a polyhedron and since $\vert\vert x \vert\vert_\infty = 1$, it follows that $\mathcal{P}(S)$ is a polytope. 
\end{proof}

%--------------
\begin{theorem}
    If $x,y \in \mathcal{C}(S)$, then $\theta(x,y) \in [0,\pi/2]$.
\end{theorem}

\begin{proof}
    If $x,y \in \mathcal{C}(S)$, then, since $\overline{y} \in \mathcal{C}(S)$, it follows that $x \circ \overline{y} \in \mathcal{C}(S)$ by part \ref{CShadamard} of \thref{hadamardcones}. Since
    \[ \langle x, y \rangle = y^\ast x = \sum_{k=1}^n \overline{y_k} \cdot x_k = \sum_{k=1}^n 1 \left( x_k \cdot \overline{y_k} \right) = e^\top (x \circ \overline{y}) = e^\ast (x \circ \overline{y}) = \langle x\circ \overline{y}, e \rangle, \]
    the result follows from \thref{SLangle}.
\end{proof}

%---------------------------------
\subsection{Basic Transformations}

As detailed in the sequel, the set $\mathcal{C}(S)$ is unchanged, or changes predictably, by certain basic transformations of $S$.

The following result is a routine exercise.

%------------
\begin{lemma}
    \thlabel{nonnegsims}
    If $A \in \mathsf{M}_n(\mathbb{R})$, $P$ is a permutation matrix, and $v > 0$, then  
        \begin{enumerate}
        [label=(\roman*)]
            \item $A \geqslant 0$ if and only if $P A P^\top \geqslant 0$; and 
            \item $A \geqslant 0$ if and only if $D_v A D_{v^{-1}} \geqslant 0$. 
        \end{enumerate}
\end{lemma}

    % \begin{proof}
    % The necessity of both the statements is clear since a finite product of nonnegative matrices is nonnegative.
    %     \begin{enumerate}
    %         [label=(\roman*)]
    %         \item To demonstrate sufficiency, suppose that the matrix $M \coloneqq P A P^\top \geqslant 0$ and, for contradiction, that $a_{ij} < 0$. If $\sigma$ is the permutation corresponding to $P$, then 
    %         \[ P = 
    %         \begin{bmatrix} 
    %         e_{\sigma(1)}^\top  \\
    %         \vdots              \\
    %         e_{\sigma(n)}^\top
    %         \end{bmatrix}. \]
    %         Setting $k = \sigma^{-1}(i)$ and $\ell = \sigma^{-1}(j)$ yields
    %         \[ m_{k\ell} = e_k^\top M e_\ell = (P^\top e_k)^\top A P^\top e_\ell = e_{\sigma(k)}^\top A e_{\sigma(\ell)} = a_{\sigma(k),\sigma(\ell)} = a_{ij} < 0, \]
    %         a contradiction. 
        
    %         \item For sufficiency, suppose that $D_v A D_v^{-1} \geqslant 0$ and, for contradiction, that $a_{ij} < 0$. Since $v > 0$, it follows that
    %         \[ \left[ D_v A D_v^{-1} \right]_{ij} = \frac{v_ia_{ij}}{v_j} < 0 \] 
    %         a contradiction.  \qedhere
    %     \end{enumerate}
    % \end{proof}

The \emph{relative gain array} was used to give a short proof of the following useful result \cite[Lemma 3.3]{jp2016}.

%------------
\begin{lemma}
    \thlabel{jplemma}
    If $P = P_\sigma$ is a permutation matrix and $x \in \mathbb{C}^n$, then $P D_x P^\top = D_{\sigma(x)}$.
\end{lemma} 

%--------------
\begin{theorem}
    \thlabel{conetransforms}
        If $S \in \mathsf{GL}_n(\mathbb{C})$, $P = P_\sigma$ is a permutation matrix, and $v$ is a totally nonzero vector, then  
        \begin{enumerate}[label=(\roman*)]
        \item \label{PS} $\mathcal{C}(PS) = \mathcal{C}(S)$; 
        \item \label{SP} $\mathcal{C}(SP) = \sigma^{-1} (\mathcal{C}(S)) \coloneqq \{ x \in \mathbb{C}^n \mid x = \sigma^{-1} (y),~ y\in \mathcal{C}(S) \}$;    
        \item \label{DvS} $\mathcal{C}(D_v S) = \mathcal{C}(S)$, $\forall v > 0$; 
        \item \label{SDv} $\mathcal{C}(S D_v) = \mathcal{C}(S)$; 
        \item \label{alphaS} $\mathcal{C}(\alpha S) = \mathcal{C}(S)$, $\forall \alpha \ne 0$;
        \item \label{Sconj} $\mathcal{C}(\bar{S}) = \overline{\mathcal{C}(S)} \coloneqq \{ y \in \mathbb{C}^n \mid y = \bar{x},~x\in \mathcal{C}(S) \}$;  
        \item \label{Sinv} $\mathcal{C}(S^{-1}) = \mathcal{C}(S^\top)$. In particular, $\mathcal{C}(S) = \mathcal{C} (S^{-\top})$, where $S^{-\top} \coloneqq (S^\top)^{-1} = (S^{-1})^\top$; and
        \item \label{Sast} $\mathcal{C}(S^\ast) = \overline{\mathcal{C}(S^{-1})}$. In particular, $\mathcal{C}(S) = \overline{\mathcal{C}(S^{-\ast})}$, where $S^{-\ast} \coloneqq (S^\ast)^{-1} = (S^{-1})^\ast$. 
        \end{enumerate}
\end{theorem}

\begin{proof}
\begin{enumerate}
    [label=(\roman*)]
        \item Follows from part (i) of \thref{nonnegsims}.
        \item If $\sigma$ is the permutation corresponding to $P$, then 
        \begin{align*}
            x \in \mathcal{C}(SP)
            &\Longleftrightarrow (SP)D_x (SP)^{-1} \geqslant 0                                    \\
            &\Longleftrightarrow S(P D_x P^\top) S^{-1} \geqslant 0                               \\
            &\Longleftrightarrow S D_y S^{-1}\geqslant 0,\ y = \sigma(x)          \tag{\thref{jplemma}}  \\
            &\Longleftrightarrow x = \sigma^{-1} (y),\ y \in \mathcal{C}(S),                       \\
            &\Longleftrightarrow x \in \sigma^{-1} \left(\mathcal{C}(S) \right). 
        \end{align*}
        \item Follows from part (ii) of \thref{nonnegsims}.
        \item Follows from the fact that 
        $$S D_x S^{-1} = S D_{v \circ x \circ v^{-1}} S^{-1} = S (D_v D_x D_{v^{-1}}) S^{-1} = (S D_v) D_x (SD_v)^{-1}.$$
        \item Immediate from part \ref{SDv} since $\alpha S = S D_{\alpha e}$.  
        \item Follows from the fact that 
        \[ \overline{S D_x S^{-1}} = \overline{S} \cdot \overline{D_x} \cdot \overline{S^{-1}} = \overline{S} \cdot D_{\bar{x}} \cdot \overline{S}^{-1}. \]
        \item Follows from the fact that $\left( S^{-1} D_x S \right)^\top = S^\top D_x S^{-\top}$. 
        \item Immediate from parts \ref{Sconj} and \ref{Sinv}. \qedhere  
\end{enumerate}
\end{proof}

%-----------------
\begin{definition}
    \thlabel{equivrel}
    If $S, T \in \mathsf{GL}_n(\mathbb{C})$, then $S$ is \emph{equivalent} to $T$, denoted by $S \sim T$, if $S = P D_v T D_w Q$, where $P = P_\sigma$ is a permutation matrix; $Q = Q_\gamma$ is a permutation matrix; $v$ is a positive vector; and $w$ is a totally nonzero vector.    
\end{definition}

%--------------
\begin{theorem}
    \thlabel{thmequivrel}
    If $\sim$ is defined as in \thref{equivrel}, then $\sim$ is an equivalence relation on $\mathsf{GL}_n(\mathbb{C})$.
\end{theorem}

%------------
\begin{proof}
If $S \in \mathsf{GL}_n(\mathbb{C})$, then it is clear that $S \sim S$.

If $S = P D_v T D_w Q$, then, by \thref{jplemma}, 
\begin{align*}
    T 
    = \left( P D_v \right)^{-1} S \left( D_w Q \right)^{-1} 
    = D_{v^{-1}} P^\top S Q^\top D_{w^{-1}} 
    = P_{\sigma^{-1}} D_{\sigma(v^{-1})} S D_{\gamma^{-1}(w^{-1})} Q_{\gamma^{-1}}. 
\end{align*}
Thus, $T\sim S$ whenever $S\sim T$.

If $S = P D_v T D_w Q$ and $T = \hat{P} D_{\hat{v}} U D_{\hat{w}} \hat{Q}$, then, by \thref{jplemma},
\begin{align*}
    S 
    = P D_v T D_w Q = (P D_v) (\hat{P} D_{\hat{v}} U D_{\hat{w}} \hat{Q}) (D_w Q) = (P\hat{P}) D_{\hat{\sigma}^{-1} (v) \circ \hat{v}} U D_{\hat{w} \circ \hat{\gamma}(w)} (\hat{Q} Q).
\end{align*}
Thus, $S\sim U$ whenever $S\sim T$ and $T \sim U$.
\end{proof}

% %------------------
% \begin{proposition}
%     If $S = \bigoplus_{k=1}^\ell S_k \in \mathsf{GL}_n(\mathbb{C})$, where $S_k \in \mathsf{GL}_{n_k}(\mathbb{C})$, then 
%     $$\mathcal{C}(S) = \prod_{k=1}^\ell \mathcal{C}(S_k)$$ and $$\mathcal{P}(S) = \prod_{k=1}^\ell \mathcal{P}(S_k).$$
% \end{proposition}

%     \begin{proof}
%     With a slight abuse of notation, for $x \in \mathbb{C}^n$, write
%     \[ x = \begin{bmatrix}
%         x_1 \\
%         \vdots \\
%         x_\ell
%     \end{bmatrix}, \]
%     where $x_k \in \mathbb{C}^{n_k}$.
%     Both assertions follow from the observation that 
%     \begin{align*}
%         S D_x S^{-1}
%         &= \bigoplus_{k=1}^\ell S_k \cdot \bigoplus_{k=1}^\ell D_{x_k} \cdot \left(\bigoplus_{k=1}^\ell S_k\right)^{-1} \\
%         &= \bigoplus_{k=1}^\ell S_k \cdot \bigoplus_{k=1}^\ell D_{x_k} \cdot \bigoplus_{k=1}^\ell S_k^{-1}              \\
%         &= \bigoplus_{k=1}^\ell \left( S_k D_{x_k} S_k^{-1} \right). \qedhere
%     \end{align*}
%     \end{proof}

%---------------------------------------
\subsection{Complex Perron Similarities}

%-----------------
\begin{definition}
    If $S \in \mathsf{GL}_n(\mathbb{C})$, then $S$ is called a \emph{Perron similarity} if there is an irreducible, nonnegative matrix $A$ and a diagonal matrix $D$ such that $A = S D S^{-1}$.
\end{definition}  

%--------------
\begin{theorem}
    \thlabel{perronsimcharacterization}
    If $S \in \mathsf{GL}_n(\mathbb{C})$, then $S$ is a Perron similarity if and only if there is a unique positive integer $k \in \bracket{n}$ such that $S e_k = \alpha x$ and $e_k^\top S^{-1} = \beta y^\top$, where $\alpha$ and $\beta$ are nonzero complex numbers such that $\alpha \beta > 0$, and $x$ and $y$ are positive vectors. Furthermore, if $S$ is a Perron similarity, then $\mathcal{C}(S)$ is nontrivial.
\end{theorem}

\begin{proof}
If $S$ is a Perron similarity, then there is an irreducible, nonnegative matrix $A$ and a diagonal matrix $D$ such that $A = S D S^{-1}$. In view of \thref{pftirr}, there are (possibly empty) diagonal matrices $\hat{D}$ and $\tilde{D}$ such that 
\[ 
D = 
\kbordermatrix{
    &           & k          &          \\
    & \hat{{D}} &            &          \\
k   &           & \rho(A)    &          \\
    &           &            & \tilde{D}
},\ 1 \leqslant k \leqslant n,
\]
where $\rho(A) \notin \sigma(\hat{D})$ and $\rho(A) \notin \sigma(\tilde{D})$. If $s_k \coloneqq S e_k$, then $As_k = \rho(A)s_k$ since $AS = SD$. Because the geometric multiplicity of an eigenvalue is less than or equal to its algebraic multiplicity \cite[p.~181]{hj2013}, it follows that $\dim \mathsf{E}_{\rho(A)} = 1$. Hence, there is a nonzero complex number $\alpha$ such that $s_k = \alpha x$, where $x$ denotes the unique right Perron vector of $A$. 

Because the line of reasoning above applies to $A^\top = (S^{-\top}) D S^\top$, it follows that there is a nonzero complex number $\beta$ such that $t_k = \beta y$, where $t_k^\top \coloneqq e_k^\top S^{-1}$ and $y$ denotes the unique left Perron vector of $A$. 

Since $S^{-1} S = I$, it follows that $1 = \left( e_k^\top S^{-1}\right)\left( S e_k \right) = t_k^\top s_k = (\alpha \beta) y^\top x$. As $x$ and $y$ are positive, we obtain $\alpha \beta = (y^\top x)^{-1} > 0$. 

Conversely, if there is a positive integer $k$ such that $S e_k = \alpha x$ and $e_k^\top S^{-1} = \beta y^\top$, where $\alpha$ and $\beta$ are nonzero complex numbers such that $\alpha \beta > 0$, and $x$ and $y$ are positive vectors, then $M_{e_k} = SD_{e_k} S^{-1} = S(e_k e_k^\top) S^{-1} = (S e_k)(e_k^\top S^{-1}) = (\alpha x)(\beta y^\top) = (\alpha \beta) xy^\top > 0$. Thus, $S$ is a Perron similarity.

For uniqueness, suppose, for contradiction, that there is a positive integer $\ell \ne k$ such that $s_{\ell} \coloneqq S e_{\ell} = \gamma u$ and $t_{\ell}^\top \coloneqq e_{\ell}^\top S^{-1} = \delta v^\top$, where $\gamma$ and $\delta$ are nonzero complex numbers such that $\gamma \delta > 0$, and $u$ and $v$ are positive vectors. Since $AS = SD$, it follows that $s_{\ell}$ and $t_{\ell}$ are right and left eigenvectors corresponding to an eigenvalue $\lambda \ne \rho(A)$. By the \emph{principle of biorthogonality} \cite[Theorem 1.4.7(a)]{hj2013}, $t_{\ell}^\ast s_k = 0$. However, 
    \begin{equation*}
    t_{\ell}^\ast s_k = \overline{t_{\ell}^\top} s_k = \overline{\delta v^\top} (\alpha x) = (\bar{\delta}\alpha) v^\top x \ne 0,
    \end{equation*}
a contradiction.

Finally, suppose that $S$ is a Perron similarity. For contradiction, if $\mathcal{C}(S)$ is trivial, then $\mathcal{A}(S)$ is trivial and only contains nonnegative matrices of the form $\alpha I$, a contradiction if $S$ is a Perron similarity.
\end{proof}

%-------------
\begin{remark}
    It is known that if $\mathcal{C}(S)$ is nontrivial, then $S$ is not necessarily a Perron similarity (see Dockter et al.~\cite[Example 11]{dppt2022}). 
\end{remark}

%-------------------------
\subsection{Normalization}
 
If $A$ is a real matrix, then any eigenvector associated with a real eigenvalue of $A$ may be taken to be real (\cite[p.~48, Problem 1.1.P3]{hj2013}), and if $v$ is an eigenvector corresponding to a nonreal eigenvalue $\lambda$ of $A$, then $\bar{v}$ is an eigenvector corresponding to $\bar{\lambda}$ (\cite[p.~45]{hj2013}). In view of these elementary facts, and taking into account parts (i)--(iv) of \thref{conetransforms} and \thref{thmequivrel}, if $S$ is a Perron similarity, then  
\begin{equation}
S \sim \begin{bmatrix} e & s_2 & \cdots & s_r & s_{r+1} & \bar{s}_{r+1} & \cdots & s_c & \bar{s}_c \end{bmatrix}, \label{strongps}
\end{equation}
where $s_i \in S^n$, $i=2,\dots,c$; $\Im(s_i) = 0$, for $i = 2,\dots,r$; and $\Im(s_i) \neq 0$, for $i=r+1,\dots,c$. Hereinafter it is assumed that every Perron similarity is normalized.

The following simple, but useful fact was shown for real matrices \cite[Lemma 2.1]{jp2017_2}. Although the proof extends to complex matrices without alteration, a demonstration is included for completeness.

%------------------
\begin{proposition}
    \thlabel{simple}
        If $S \in \mathsf{GL}_n(\mathbb{C})$, then $x^\top S^{-1} \geqslant 0$ if and only if $x \in \mathcal{C}_r(S)$. %Furthermore, $x^\top S^{-1} > 0$ if and only if $x \in \interior(\mathcal{C}_r(S))$
\end{proposition}

\begin{proof}
Notice that 
\begin{equation*}
    y^\top \coloneqq x^\top S^{-1} \geqslant 0
    \Longleftrightarrow x^\top = y^\top S,~y \geqslant 0             
    \Longleftrightarrow x \in \mathcal{C}_r(S). \qedhere
\end{equation*} 
% and 
% \begin{equation*}
%     y^\top \coloneqq x^\top S^{-1} > 0
%     \Longleftrightarrow x^\top = y^\top S,~y > 0             
%     \Longleftrightarrow x \in \interior(\mathcal{C}_r(S)).   \qedhere
% \end{equation*}
\end{proof}

% An immediate consequence of \thref{simple} is the following result, which was established for real matrices \cite[Theorem 3.11]{jp2016} although the preceding argument result applies to real matrices. 

% %----------------
% \begin{corollary}
%     \thlabel{psequiv}
%         If $S \in \mathsf{GL}_n(\mathbb{C})$, then $e_k^\top S^{-1} \geqslant 0$ if and only if $e_k \in \mathcal{C}_r(S)$. Moreover, $e_k^\top S^{-1} > 0$ if and only if $e_k \in \interior(\mathcal{C}_r(S))$. 
% \end{corollary}

% %--------------
% \begin{theorem}
%     If $S \in \mathsf{GL}_n(\mathbb{C})$, then $S$ is a Perron similarity if and only if there is a unique positive integer $k \in \bracket{n}$ such that $e_k^\top \in \interior(\mathcal{C}_r(S))$ and $e_k^\top \in \interior(\mathcal{C}_r(S^{-\top}))$.
% \end{theorem}

% \begin{proof}
%     Immediate from \thref{perronsimcharacterization} and \thref{psequiv}.
% \end{proof}

\begin{definition}
    Given an $m$-by-$n$ matrix $S$, the \emph{row cone of $S$} \cite{jp2017_2} denoted by $\mathcal{C}_r(S)$, is the the polyhedral cone generated by the rows of $S$, i.e., $\mathcal{C}_r(S)\coloneqq \coni(r_1,\dots,r_m)$ and the \emph{row polytope of $S$}, denoted by $\mathcal{P}_r(S)$, is the polytope generated by the rows of $S$, i.e., $\mathcal{P}_r(S) \coloneqq \conv(r_1,\dots, r_m)$.     
\end{definition}

Johnson and Paparella \cite{jp2016} demonstrated that $\mathcal{C}(S)$ can coincide with $\mathcal{C}_r(S)$ for a class of \emph{Hadamard matrices} called \emph{Walsh matrices} (see \thref{walshmatrices} below). We extend and generalize these results for complex matrices. 

%-----------------
\begin{definition}
    If $S$ is a Perron similarity, then $S$ is called \emph{ideal} if $\mathcal{C}(S) = \mathcal{C}_r(S)$.
\end{definition}

% %------------------
% \begin{proposition}
%     \thlabel{rhcrows}
%         If $S \in \mathsf{GL}_n(\mathbb{C})$, then $r_i \in \mathcal{C}(S)$ if and only if $r_i \circ r_j \in \mathcal{C}_r(S),\ \forall j \in \bracket{n}$.
% \end{proposition}

% \begin{proof}
% Because the $j$\textsuperscript{th}-row of $S D_{r_i}$ is $r_j \circ r_i$, it follows that 
% \begin{align*}
% r_i \in \mathcal{C}(S) 
% & \Longleftrightarrow S D_{r_i} S^{-1} \geqslant 0 											\\ 
% & \Longleftrightarrow (S D_{r_i}) S^{-1} \geqslant 0 										\\
% & \Longleftrightarrow \left( r_j \circ r_i \right) S^{-1} \geqslant 0,~\forall j \in \bracket{n} 					\\
% & \Longleftrightarrow r_j \circ r_i = r_i \circ r_j \in \mathcal{C}_r(S),~\forall j \in \bracket{n} \tag{\thref{simple}}
% \end{align*}
% as desired.
% \end{proof}

% Of interest will be Perron similarities whose rows are all realizable. Thus, we consider the following class of matrices.

% %------------
% \begin{definition}
%     \thlabel{rhc}
%         If $S \in \mathsf{M}_{m \times n}(\mathbb{C})$, then $S$ is called \emph{row Hadamard conic} (RHC) if \( r_i \circ r_j \in \mathcal{C}_r(S),~\forall i,j \in \bracket{m} \).
% \end{definition}

% %--------------
% \begin{theorem}
%     \thlabel{rhcequiv}
%         If $S \in \mathsf{GL}_{n}(\mathbb{C})$, then $r_i \in \mathcal{C}(S),\ \forall i \in \bracket{n}$ if and only if $S$ is RHC.
% \end{theorem}

% \begin{proof}
% Immediate from \thref{rhcrows} and \thref{rhc}.
% \end{proof}

%----------
\begin{lemma}
    \thlabel{realizablerows}
        If $S \in \mathsf{GL}_n(\mathbb{C})$, then $\mathcal{C}_r(S) \subseteq \mathcal{C}(S)$ if and only if $r_i \in \mathcal{C}(S),\ \forall i \in \bracket{n}$.
\end{lemma}

\begin{proof}
Immediate from \thref{conecontain}.
\end{proof}

%----------
\begin{lemma}
    \thlabel{allonesrow}
        If $S \in \mathsf{GL}_n(\mathbb{C})$, then \(\mathcal{C}(S) \subseteq \mathcal{C}_r(S)\) if and only if \( e \in \mathcal{C}_r(S)\).
\end{lemma}

\begin{proof}
The necessity of the condition is trivial given that $e \in \mathcal{C}(S)$. 

For sufficiency, assume that \( e \in \mathcal{C}_r(S)\) and let $x \in \mathcal{C}(S)$. By definition, there is a nonnegative vector $y$ such that $e^\top = y^\top S$ and $S D_x S^{-1} \geqslant 0$. Since 
\[ x^\top S^{-1} = (e^\top D_x) S^{-1} = ((y^\top S) D_x) S^{-1} = y^\top (S D_x S^{-1}) \geqslant 0, \]
it follows that $x \in \mathcal{C}_r(S)$ by \thref{simple} .
\end{proof}  

%--------------
\begin{theorem}
    \thlabel{thm:idealsims}
        If $S$ is a Perron similarity, then $S$ is ideal if and only if \( e \in \mathcal{C}_r(S)\) and $r_i \in \mathcal{C}(S),\ \forall i \in \bracket{n}$.
\end{theorem}

\begin{proof}
    Immediate from Lemmas \ref{realizablerows} and \ref{allonesrow}.
\end{proof}

%--------------
\begin{theorem}
    \thlabel{coneandpoly}
    If $S$ is a Perron similarity, then $S$ is ideal if and only if $\mathcal{P}_r(S) = \mathcal{P}(S)$.
\end{theorem}

\begin{proof}
If $S$ is ideal, then $\mathcal{C}_r(S) = \mathcal{C}(S)$. If $x \in \mathcal{P}_r(S)$, then $x \in \mathcal{C}_r(S)$ and, by hypothesis, $x \in \mathcal{C}(S)$. Thus, $M_x \geqslant 0$ and it suffices to demonstrate that $M_x e = e$. Since $Se_1 = e$, it follows that $S^{-1} e = e_1$. Furthermore, any convex combination of the rows of $S$ produces a vector whose first entry is 1. Thus, $x_1 = 1$ and 
\begin{align*}
    M_x e = (S D_x S^{-1}) e = S D_x e_1 = S e_1 = e,
\end{align*}
i.e., $x \in \mathcal{P}(S)$. If $x \in \mathcal{P}(S)$, then $M_x \geqslant 0$ and $M_x e = e$. Since $Se_1 = e$ and $M_x$ has a positive eigenvector, it follows that $x_1 = \rho(M_x) = 1$ \cite[Corollary 8.1.30]{hj2013}. Notice that $x \in \mathcal{P}(S) \implies x \in \mathcal{C}(S) \implies x \in \mathcal{C}_r(S)$, i.e., $\exists y \geqslant 0$ such that $x^\top = y^\top S$. Thus, 
$$y^\top e = (x^\top S^{-1}) e = x^\top (S^{-1} e) = x^\top e_1 = x_1 = 1,$$ 
i.e., $x \in \mathcal{P}_r(S)$.

Conversely, suppose that $\mathcal{P}_r(S) = \mathcal{P}(S)$. Since $e \in \mathcal{C}(S)$, it follows from Propositions \ref{exptball} and \ref{exptconv} that one of the rows of $S$ must be $e^\top$. Thus, $\mathcal{C}(S) \subseteq \mathcal{C}_r(S)$ by \thref{allonesrow}. By hypothesis, every row of $S$ is realizable so that, following {Lemma} \ref{realizablerows}, $\mathcal{C}_r(S) \subseteq \mathcal{C}(S)$.
\end{proof}

%----------------
\begin{corollary}
If $S$ is a Perron similarity, then $S$ is ideal if and only if $r_i \in \mathcal{C}(S)$, $\forall i \in \bracket{n}$, and $\exists k \in \bracket{n}$ such that $e_k^\top S = e^\top$. 
\end{corollary}

\begin{proof}
Immediate from \thref{cor:allonesextreme} and \thref{thm:idealsims}.
\end{proof}

% %-------------
% \begin{remark}
%     \thlabel{dft_remark}
%     It is clear (or otherwise easy to show) that if $S$ is a Perron similarity and $P$ is a permutation matrix, then $S$ is ideal if and only if $PS$ is ideal.
% \end{remark}

%-----------------
\begin{definition}
    If $S$ is a Perron similarity, then $S$ is called \emph{extremal} if $\mathcal{P}(S)$ contains an extremal point other than $e$.
\end{definition}

%-----------------------------------------------
\subsection{Kronecker Product \& Walsh Matrices}
 
If $A \in \mathsf{M}_{m \times n}(\mathbb{F})$ and $B \in \mathsf{M}_{p \times q}(\mathbb{F})$, then the \emph{Kronecker product of $A$ and $B$}, denoted by $A \otimes B$, is the $mp$-by-$nq$ matrix defined blockwise by $A \otimes B = \begin{bmatrix} a_{ij}B \end{bmatrix}$. 

If $A \in \mathsf{M}_{m \times n}(\mathbb{F})$ and $p \in \mathbb{N}_0$, then 
\[ 
A^{\otimes p} \coloneqq
\begin{cases}
    [1], & p = 0 \\
    A^{\otimes (p-1)} \otimes A = A \otimes A^{\otimes (p-1)}, & p > 1.
\end{cases} \]

Although some of the definitions differ, the proofs for the following results, established by Dockter et al.~\cite{dppt2022}, can be retooled to obtain the following.

%--------------
\begin{theorem}
    [{\cite[Theorem 7]{dppt2022}}]
        \thlabel{kronPS}
        If $S$ and $T$ are Perron similarities, then $S\otimes T$ is a Perron similarity. 
\end{theorem}

%--------------
\begin{theorem}
    [{\cite[Theorem 13]{dppt2022}}]
        \thlabel{kronideal}
        If $S$ and $T$ are ideal, then $S\otimes T$ is ideal. 
\end{theorem}

%-----------------
\begin{definition}
    If $H = [h_{ij}] \in \mathsf{M}_n(\{ \pm 1 \})$, then $H$ is called a \emph{Hadamard matrix} if $HH^\top = nI_n$.    
\end{definition}

%-----------------
\begin{definition}
    \thlabel{walshmatrices}
    If $n \in \mathbb{N}_0$, then the matrix
    \[
    H_{2^n} \coloneqq
    \begin{bmatrix}
        1 & 1   \\
        1 & -1
    \end{bmatrix}^{\otimes n}    
    \]
    is called \emph{Sylvester's Hadamard matrix} or, for brevity, the \emph{Walsh matrix (of order $2^n$)}. 
\end{definition}

It is well-known that $H_{2^n}$ is a Hadamard matrix. Notice that $H_1 = [1]$, $H_2 = \left[\begin{array}{rr} 1 & 1 \\ 1 & -1 \end{array}\right]$
and 
\[ H_4 =
\left[ \begin{array}{*{4}{r}}
    1 & 1 & 1 & 1 \\
    1 & -1 & 1 & -1 \\
    1 & 1 & -1 & -1 \\
    1 & -1 & -1 & 1
\end{array} \right]. \] 
The theory of \emph{association schemes} was used to prove that these matrices are ideal  \cite[Proposition 5.1 and Theorem 5.2]{jp2016}. However, a proof of this is readily available via induction coupled with Theorems \ref{kronPS} and \ref{kronideal}.

%--------------
\begin{theorem}
    If $n\in \mathbb{N}_0$, then $H_{2^n}$ is ideal, extremal, and $\dim(\mathcal{C}(H_{2^n}))=2^n$.
\end{theorem}

Although every normalized Hadamard matrix is a Perron similarity, not every Hadamard matrix is ideal (it can be verified via the MATLAB-command \verb|hadamard(12)| that only the first row, i.e., the all-ones vector, of the normalized Hadamard matrix of order twelve is realizable). However, it is known that if $H$ is a Hadamard matrix, then $\conv(e,e_1-e_2,\ldots,e_1-e_n) \subseteq \mathcal{P}(H)$ \cite[Remark 6.4]{jp2016}. Furthermore, in view of \thref{allonesrow}, we have 
\[ \conv(e,e_1-e_2,\ldots,e_1-e_n) \subseteq \mathcal{P}(H) \subseteq \mathcal{P}_r(H). \]
Thus, every Hadamard matrix is extremal.

If $x \in \mathbb{C}^n$ and $\sigma_1,\ldots, \sigma_n \in \mathsf{Sym}(n)$, then a matrix of the form 
\[ X = 
\begin{bmatrix}
    \sigma_1(x)^\top \\
    \vdots          \\
    \sigma_n(x)^\top
\end{bmatrix} \]
is called a \emph{permutative matrix}. Notice that permutation matrices and circulant matrices are permutative matrices. 

If $y = Hx$ and $M_y = 2^{-n} H_{2^n} D_y H_{2^n}$, then $A$ is a permutative matrix \cite[p.~295]{jp2016}. For example, if $n=2$, then
\[ M_y = 
\begin{bmatrix}
    x_1 & x_2 \\
    x_2 & x_1
\end{bmatrix}
\]
and when $n=4$, 
\begin{equation} 
    \label{kleinfour} 
    M_y = 
    \begin{bmatrix}
        x_1 & x_2 & x_3 & x_4 \\
        x_2 & x_1 & x_4 & x_3 \\
        x_3 & x_4 & x_1 & x_2 \\
        x_4 & x_3 & x_2 & x_1
    \end{bmatrix}.
\end{equation}

In general, let 
\[ 
P_{1}^{(1)} := 
\begin{bmatrix}
1 & 0 \\
0 & 1
\end{bmatrix} \mbox{ and }
P_{2}^{(1)} := 
\begin{bmatrix}
0 & 1 \\
1 & 0
\end{bmatrix}. \] 
and for $n \geq 2$, let 
\[ P_{n,k} :=
\left\{ 
\begin{array}{rl} 
\begin{bmatrix} 
P_{(n-1),k} & 0             \\
0           & P_{(n-1),k} 
\end{bmatrix} \in \mathsf{M}_{2^n}(\mathbb{R}), & k \in \bracket{2^{n-1}}						\\ \\ 
\begin{bmatrix} 
0                       & P_{(n-1),k - 2^{n-1}} \\ 
P_{(n-1), k - 2^{n-1}}  & 0 
\end{bmatrix} \in \mathsf{M}_{2^n}(\mathbb{R}), & k \in \bracket{2^n}\backslash \bracket{2^{n-1}},
\end{array} \right. \]
If $y = Hx$, then 
\begin{equation}
    \label{genkleinmatrix}
M_y = 
\begin{bmatrix}
    x^\top P_{n,1} \\
    \vdots      \\
    x^\top P_{n,2^k} 
\end{bmatrix}.
\end{equation}
Kalman and White \cite{kw2001} called the matrix \eqref{kleinfour} a \emph{Klein matrix}. As such, we call any matrix of the form given in \eqref{genkleinmatrix} a \emph{Klein matrix of order $2^n$}. 

%---------------------------
\section{Circulant Matrices}
%---------------------------

In what follows, we let $F = F_n$ denote the $n$-by-$n$ \emph{discrete Fourier transform matrix}, i.e., $F$ is the $n$-by-$n$ matrix such that 
\[ f_{ij} = \frac{1}{\sqrt{n}} \omega^{(i-1)(j-1)}, \]
with 
$$\omega = \omega_n \coloneqq \cos \left( \frac{2\pi}{n} \right) - \sin \left( \frac{2\pi}{n} \right)\ii.$$ 
If $n > 1$, then 
\begin{align*}
    F 
    &=\frac{1}{\sqrt{n}}
    \kbordermatrix{
            & 1      & 2            & \cdots & k                    & \cdots & n                    \\
    1       & 1      & 1            & \cdots & 1                    & \cdots & 1                    \\
    2       & 1      & \omega       & \cdots & \omega^{k-1}         & \cdots & \omega^{n-1}         \\
    \vdots  & \vdots & \vdots       & \ddots & \vdots               &        & \vdots               \\
    k       & 1      & \omega^{k-1} & \cdots & \omega^{(k-1)^2}     & \cdots & \omega^{(k-1)(n-1)}  \\
    \vdots  & \vdots & \vdots       &        & \vdots               & \ddots & \vdots               \\
    n       & 1      & \omega^{n-1} & \cdots & \omega^{(k-1)(n-1)}  & \cdots & \omega^{(n-1)(n-1)} 
    }                                                                                                       \\
    &= \frac{1}{\sqrt{n}}
    \kbordermatrix{
        & 1 & 2         & \cdots & k                    & \cdots    & n             \\
        & e & v_\omega  & \cdots & v_\omega^{k-1}       & \cdots    & v_\omega^{n-1} 
   },  
\end{align*} 
with
\[ v_\omega \coloneqq  
    \kbordermatrix{
            &             \\
    1       & 1           \\
    2       & \omega      \\
    \vdots  & \vdots      \\
    k       & \omega^{k-1} \\
    \vdots  & \vdots       \\
    n       & \omega^{n-1}  
    } . \]
Because $F$ is symmetric and unitary \cite[Theorem 2.5.1]{d1979}, it follows that $F^{-1} = F^\ast = \overline{F^\top} = \overline{F}$.

Since $\omega^k \cdot \omega^{n-k} = 1 = \omega^k \cdot \overline{\omega^k}$, it follows that $\omega^{n-k} = \overline{\omega^k}$. Consequently, 
\begin{equation}
    \label{vomega}
    v_\omega^{n-k} = \overline{v_w^k},\ 1 \leqslant k \leqslant n-1.  
\end{equation}
Thus, 
$$F = \frac{1}{\sqrt{n}}
\kbordermatrix{
    & 1 & 2         & 3          &  \cdots  & n-1               & n                 \\
    & e & v_\omega  & v_\omega^2 & \cdots   & \overline{v_w^2}  & \overline{v_w} 
}.$$

%-----------------
\begin{definition}
    [{\cite[p.~66]{d1979}}]
    \label{circulant-def-1}
    If $c = \begin{bmatrix} c_1 & \cdots & c_n \end{bmatrix}^\top \in \mathbb{C}^n$ and $C \in \mathsf{M}_n(\mathbb{C})$ is a matrix such that
        \[ c_{ij} = c_{((j-i)\bmod{n}) + 1},\ \forall (i,j) \in \bracket{n}^2, \]
    then $C$ is called a \emph{circulant} or a \emph{circulant matrix} with \emph{reference vector} $c$. In such a case, we write $C = \circulant(c) = \circulant(c_1,\ldots,c_n)$. 
\end{definition}

If $C \in \mathsf{M}_n(\mathbb{C})$, then $C$ is circulant if and only if there is a diagonal matrix $D$ such that $C = FDF^\ast = FD\overline{F}$ \cite[Theorems 3.2.2 and 3.2.3]{d1979}.

Recall that $A=[A_{ij}]$, with $A_{ij} \in \mathsf{M}_n(\mathbb{C})$ is called (an $m$-by-$m$) \emph{block matrix}. If $C =  \circulant(C_1,\ldots,C_m)$ with $C_k \in \mathsf{M}_n(\mathbb{C}),\ \forall k \in \bracket{m}$, then the block matrix $C$ is called \emph{block-circulant} or a \emph{block-circulant matrix} of type $(m,n)$ \cite[\S 5.6]{d1979}. The set of block-circulant matrices of type $(m,n)$ is denoted by $\mathcal{BC}_{mn}$. %$\mathscr{BC}_{(m,n)}$. 

If $C = [C_{ij}]$ is an $m$-by-$m$ block matrix and $C_{ij}$ is circulant, then $C$ is called a \emph{circulant block matrix} and the set of such matrices is denoted by $\mathcal{CB}_{mn}$ \cite[\S 5.7]{d1979}.     

Combining the previous sets yields the class of \emph{block circulant matrices with circulant blocks}, i.e., block matrices of the form $C =  \circulant(C_1,\ldots,C_m)$, where $C_1,\ldots,C_m$ are circulant. The set of such matrices is denoted by $\mathcal{BCCB}_{mn}$. All matrices in $\mathcal{BCCB}_{mn}$ are simultaneously diagonalizable by the unitary matrix $F_m \otimes F_n$ and any matrix of the form 
\[ (F_m\otimes F_n) D (\overline{F_m}\otimes \overline{F_n}), \]
with $D$ diagonal, belongs to $\mathcal{BCCB}_{mn}$ \cite[Theorem 5.8.1]{d1979}. 

%------------------------------------------------
\section{Perron Similarities arising from K-arcs}
%------------------------------------------------

In this section, we examine Perron similarities from realizing matrices of points on Type I K-arcs. Attention is focused on Type I K-arcs because many Type II arcs and Type III arcs are pointwise powers of Type I arcs (see, e.g., Munger et al.~\cite[Section 5]{mnp2024}).

%------------------
\subsection{Type 0} 

Points on the Type 0 arc are zeros of the reduced Ito polynomial $f_\alpha(t) = (t  - \beta)^n - \alpha^n$, where $\alpha \in [0,1]$ and $\beta \coloneqq 1 - \alpha$. In \cite{jp2017} it was showed (and it is easy to verify otherwise) that the matrix 
\[ 
M = M(\alpha) \coloneqq 
\begin{bmatrix}
\beta & \alpha & 	 	\\
 & \beta & \alpha 		\\
 & & \ddots & \ddots 	\\
 & & & \beta & \alpha 	\\
\alpha & & & & \beta
\end{bmatrix} \]
realizes this arc, i.e., $\chi_M = f_\alpha$. Because $M$ is a circulant, there is a diagonal matrix $D$ such that $M = F D F^\ast$. As $F$ is a scaled Vandermonde matrix, we defer its discussion to the subsequent section. 

%------------------
\subsection{Type I}
\label{type1} 

For Type I arcs, the reduced Ito polynomial is $f_\alpha(t) = t^{s}  - \beta  t^{s-q} - \alpha$, with $\alpha \in [0,1]$ and $\beta \coloneqq 1 -\alpha$. If 
\[ M = M(\alpha) \coloneqq 
\begin{bmatrix}
0 & I \\
\alpha & \beta e_{s-q}^\top 
\end{bmatrix}, \label{typeonemats} \]
then $M$ is a nonnegative irreducible companion matrix and $\chi_M = f_\alpha$. Following Remark \ref{rem:distinctroots}, there are at most $2s-1$ complex values, and hence at most $2s-1$ values in $[0,1]$, such that $f_\alpha$ does not have distinct roots. Notice that $f_\alpha(1) = 0$ and if $1, \lambda_2, \dots,\lambda_s$ are the distinct zeros of $f_\alpha$, then the Vandermonde matrix 
\begin{equation}
S = S(\alpha) =
\begin{bmatrix}
1 & 1 & \cdots & 1                  			\\
1 & \lambda_2 & \cdots & \lambda_s              \\
\vdots & \vdots & \ddots & \vdots   			\\
1 & \lambda_2^{s-1} & \cdots & \lambda_s^{s-1} 
\end{bmatrix} \in \mathsf{GL}_{s}(\mathbb{C})				\label{vandermonde}
\end{equation}
is a Perron similarity satisfying $M = S \diag{\left( 1,\lambda_2,\dots,\lambda_s \right)} S^{-1}$. The Perron similarity $S$ is ideal because every row is realizable (the $k\textsuperscript{th}$-row forms the spectrum of $M^k \geqslant 0$) and the first row is $e$. Furthermore, it is extremal because the second row is extremal.

%----------------
\begin{corollary}
    \thlabel{TypeIsims}
    Let $f_\alpha$ be defined as in \eqref{polyalpha}, and let $S = S(\alpha)$ be the Vandermonde matrix defined as in $\eqref{vandermonde}$ corresponding to the zeros of $f_\alpha$. 
        \begin{enumerate}
        [label=(\roman*)]
        \item If $n$ is even, then $S$ is ideal and extremal.
        
        \item If $n$ is odd and $\alpha \geqslant \beta$, then $S$ is ideal and extremal.
        
        \item If $n$ is odd, $\alpha < \beta$, and $(n-1)^{n-1} \beta^n - n^n \alpha^{n-1} \neq 0$, then  $S$ is ideal and extremal.   
        \end{enumerate}
\end{corollary}

%-------------------------------------------
\section{Circulant and Block-Circulant NIEP}
%-------------------------------------------

Since $S \sim \alpha S$, with a slight abuse of notation we let $F = F_n = [f_{ij}]$ denote the $n$-by-$n$ matrix such that $f_{ij} = \omega^{(i-1)(j-1)}$. Notice that $F$ is a Vandermonde matrix corresponding to the $n$ distinct $n\textsuperscript{th}$-roots of unity, i.e, 
\begin{align*} 
F
&=
\kbordermatrix{
        & 1      & 2            & \cdots & k                    & \cdots & n                    \\
1       & 1      & 1            & \cdots & 1                    & \cdots & 1                    \\
2       & 1      & \omega       & \cdots & \omega^{k-1}         & \cdots & \omega^{n-1}         \\
\vdots  & \vdots & \vdots       & \ddots & \vdots               &        & \vdots               \\
k       & 1      & \omega^{k-1} & \cdots & \omega^{(k-1)^2}     & \cdots & \omega^{(k-1)(n-1)}  \\
\vdots  & \vdots & \vdots       &        & \vdots               & \ddots & \vdots               \\
n       & 1      & \omega^{n-1} & \cdots & \omega^{(k-1)(n-1)}  & \cdots & \omega^{(n-1)^2} 
} \\
&=
\kbordermatrix{
& 1 & 2         & 3          &  \cdots  & n-1               & n                 \\
& e & v_\omega  & v_\omega^2 & \cdots   & \overline{v_w^2}  & \overline{v_w} 
},\ n > 1.
\end{align*}
It is easy to show that
\begin{equation}
    \label{Finv}
    F^{-1} = \frac{1}{n} \overline{F}, 
\end{equation}
\begin{equation}
    \label{FmFninv}
    (F_m \otimes F_n)^{-1} = \frac{1}{mn} \left( \overline{F_m} \otimes \overline{F_n} \right),    
\end{equation}
and $C \in \mathsf{M}_n(\mathbb{C})$ is circulant if and only if there is a diagonal matrix $D$ such that $C = FDF^{-1}$. 

Furthermore, if $y \in \mathbb{C}^n$, $x = Fy$, and $M_x = F D_x F^{-1}$, then, following \eqref{Mijentry},  
\begin{equation}
    \label{M1jentry}
    [M_x]_{1,j} 
    = \left\langle x, \overline{e \circ \frac{\overline{f_j}}{n}} \right\rangle 
    = \left\langle x, \frac{f_j}{n} \right\rangle 
    =\frac{1}{n}\overline{f_j}^\top (Fy)
    = y_j,\ \forall j \in \bracket{n}
\end{equation} 
i.e., $M_x = \circulant y$. 

%----------------
\begin{corollary}
    [extreme ray and vertex representation]
    \thlabel{dftconetope}
    If $n \in \mathbb{N}$, then $F$ is ideal, extremal, and $\dim \mathcal{C}(F) = n$.
\end{corollary}

\begin{proof}
    The claim that $F$ is ideal is immediate from \thref{TypeIsims} or \eqref{M1jentry}. Clearly, $F$ is extremal since every entry of $F$ is extremal in $\Theta_n$. Finally, $\dim \mathcal{C}(F) = n$ because $F$ is invertible.        
\end{proof}

%--------------
\begin{theorem}
    [half-space description]
    If $n$ is a positive integer, then
    $$\mathcal{C}(F) = \bigcap_{k \in \bracket{n}} \mathsf{H}_k,$$
    where 
    $$\mathsf{H}_k \coloneqq \mathsf{H}({f_k}) \cap \mathsf{H}(\ii {f_k}) \cap \mathsf{H}(-\ii {f_k})$$
    and $f_k \coloneqq F e_k$, $k \in \bracket{n}$.
\end{theorem}

\begin{proof}
    If $x \in \mathcal{C}(F)$, then, by \thref{dftconetope}, there is a nonnegative vector $y$ such that $x^\top = y^\top F$. Since $F$ is symmetric, it follows that $x = Fy$. Since  
    \begin{equation}
        \label{yk}
        \langle x, f_k \rangle 
        = \left( F e_k \right)^\ast (F y)             
        = \left( e_k^\top  \overline{F} \right) (F y) 
        = e_k^\top \left(\left( \overline{F}F \right) y \right)  
        = e_k^\top ((n I_n)y)                  
        = n y_k, 
    \end{equation}
    it follows that $\Re(\langle x, {f_k} \rangle) \geqslant 0$ and $x \in \mathsf{H}(f_k)$. Because $\langle x, \pm \ii f_k \rangle = \mp \ii \langle x, f_k \rangle = \mp \ii n y_k$, it follows that $\Re (\langle x, \pm \ii f_k \rangle ) \geqslant 0$, i.e., $x \in \mathsf{H}(\pm \ii f_k)$. As $k$ and $x$ were arbitrary, $\mathcal{C}(F) \subseteq \bigcap_{k \in \bracket{n}} \mathsf{H}_k$.
    
    Let $x \in \bigcap_{k \in \bracket{n}} \mathsf{H}_k$ and suppose, for contradiction, that $x \notin \mathcal{C}(F)$. Because $F$ is invertible, there is a unique complex vector $y$ such that $x = Fy$, and since $F$ is symmetric, we have $x^\top = y^\top F$. As $x \notin \mathcal{C}(F)$, it follows that $a \coloneqq \Re y \not \geqslant 0$ or $b \coloneqq \Im y \ne 0$. Thus, $\exists k \in \bracket{n}$ such that $\Re y_k = a_k < 0$ or $\Im y_k = b_k \ne 0$. Notice that $\langle x, f_k \rangle = ny_k,\ \forall k \in \bracket{n}$ because the calculations in \eqref{yk} do not rely on the nonnegativity of $y$. If $a_k < 0$, then $\langle x, f_k \rangle = n y_k = n(a_k + \ii b_k)$ implies $\Re (\langle x, f_k \rangle) = n a_k < 0$, a contradiction since $x \in \mathsf{H}(f_k)$. Otherwise, if $b_k \ne 0$, then $b_k < 0$ or $b_k > 0$; if $b_k < 0$, then the equation 
    $$\langle x, \ii f_k \rangle = -\ii \langle x, f_k \rangle = -\ii n y = n(b_k - \ii a_k)$$ 
    implies $\Re (\langle x, \ii f_k \rangle) = n b_k < 0$, a contradiction since $x \in \mathsf{H}(\ii f_k)$; if $b_k > 0$, then the equation 
    $$\langle x, -\ii f_k \rangle = \ii \langle x, f_k \rangle = \ii n y = n(-b_k + \ii a_k)$$ 
    implies $\Re (\langle x, -\ii f_k \rangle) = -n b_k < 0$, a contradiction since $x \in \mathsf{H}(-\ii f_k)$. Thus, $\mathcal{C}(F) \supseteq \bigcap_{k \in \bracket{n}} \mathsf{H}_k$ and the demonstration is complete.
\end{proof}

%------------------
\begin{proposition}
    \thlabel{propF}
    If $n \in \mathbb{N}$, then:
        \begin{enumerate}
            [label=(\roman*)]
            \item \label{FequalsconjF} $\mathcal{C}(F) = \mathcal{C}\left(\overline{F}\right)$;
            \item \label{barFideal} $\overline{F}$ is ideal; and 
            \item $\mathcal{P}(\overline{F}) = \mathcal{P}_r(\overline{F})$.
        \end{enumerate}
\end{proposition}

\begin{proof}
    \begin{enumerate}
        [label=(\roman*)]
        \item Notice that
        \begin{align*}
            \mathcal{C}\left(\overline{F}\right) 
            = \mathcal{C}\left( \frac{1}{n} \overline{F}\right)    
            = \mathcal{C}(F^{-1})                                  
            = \mathcal{C}(F^\top)                                 
            = \mathcal{C}(F).                                      
        \end{align*}

        \item Follows from the observation that
        \begin{align*}
            y \in \mathcal{C}\left(\overline{F}\right) 
            &\iff y \in \overline{\mathcal{C}(F)}           \\
            &\iff y = \overline{x},\ x \in \mathcal{C}(F)   \\
            &\iff y = \overline{x},\ x = Fz,\ z \geqslant 0       \\
            &\iff y = \overline{F}z,\ z\geqslant 0                \\
            &\iff y \in \mathcal{C}_r(\overline{F}).
        \end{align*}
        
        \item The assertion that $\mathcal{P}(\overline{F}) = \mathcal{P}_r(\overline{F})$ is an immediate consequence of part \ref{barFideal} and \thref{coneandpoly}. \qedhere
    \end{enumerate}
\end{proof}

There is an easier test for feasibility, as follows.

%--------------
\begin{theorem}
    If $x \in \mathbb{C}^n$, then $\Lambda(x)$ is realizable by an $n$-by-$n$ circulant matrix if and only if $F x \geqslant 0$.    
\end{theorem}

\begin{proof}
    If $\Lambda(x)$ is realizable by an $n$-by-$n$ circulant matrix $C$, then $x \in \mathcal{C}(F)$. By parts \ref{FequalsconjF} and \ref{barFideal} of \thref{propF}, there is a nonnegative vector $y$ such that $x^\top = y^\top \overline{F}$. Note that $x = \overline{F}y$ since $\overline{F}$ is symmetric and because $F\overline{F} = n I_n$, we have $Fx = ny \geqslant 0$.  

    Conversely, if $y \coloneqq Fx \geqslant 0$, then 
    $$x = \frac{1}{n}\overline{F}y$$ 
    and since $\overline{F}$ is symmetric, it follows that  
    $$ x^\top = \frac{1}{n}y^\top \overline{F}.$$ 
    By parts \ref{FequalsconjF} and \ref{barFideal} of \thref{propF}, $x \in \mathcal{C}({G})$, i.e., $\Lambda(x)$ is realizable by an $n$-by-$n$ circulant matrix. 
\end{proof}

% \begin{proof}
%     Let $P$ be the permutation matrix defined as in equation \eqref{special_perm_matrx}. 
    
%     If $\Lambda(x)$ is the spectrum of an $n$-by-$n$ circulant matrix $C$, then $C = G D_x F^{-1}$, i.e., $x \in \mathcal{C}(F)$. By \thref{dftconetope}, there is a nonnegative vector $y$ such that $x^\top = y^\top G$. Since
%     \begin{align*}
%         x^\top = y^\top G
%         &\iff x = G^\top y                  \\
%         &\iff x = Gy                        \\
%         &\iff y = F^{-1} x                  \\
%         &\iff y = \frac{1}{n} \overline{F}x \\
%         &\iff ny = PG x                     \\
%         &\iff nP^\top y = Gx,
%     \end{align*}  
%     it follows that $Gx \geqslant 0$ whenever $y \geqslant 0$.

%     Conversely, suppose that $y = Gx \geqslant 0$. If $z = Py$, then $z \geqslant 0$ and $$z = Py = P(Gx) = (PG)x = \overline{F} x = n F^{-1}x.$$ Thus, $x = \frac{1}{n}G z$, and since $G = G^\top$, we have $x^\top = \frac{1}{n} z^\top G$, i.e., $x \in \mathcal{C}_r(G)$. By \thref{dftconetope}, $x \in \mathcal{C}(F)$. Thus, $C = G D_x F^{-1}$ is a nonnegative circulant matrix with spectrum $\Lambda(x)$. 
% \end{proof}

%----------------
\begin{corollary}
    If $x \in \mathbb{C}^n$, then $\Lambda(x)$ is realizable by an $n$-by-$n$ circulant matrix if and only if 
    \begin{equation}
        \label{circulant_inequalities}
        \sum_{j=1}^n \omega^{(i-1)(j-1)} x_j \geqslant 0,\ \forall i \in \bracket{n}.
    \end{equation}
\end{corollary}

If $x \in \mathbb{C}^n$ satisfies the inequalities in \eqref{circulant_inequalities}, then the inequality corresponding to $i=1$ in yields 
\[ \sum_{k=1}^n x_k \geqslant 0 \]
corresponding to the trace-condition.

%-------------
\begin{remark}
    \thlabel{rem_symm}
    For $n \geqslant 2$, let 
    $$\mathbb{C}_{\rm sym}^n \coloneqq \{ x \in \mathbb{C}^n \mid \Im x_1 = 0,\ \overline{x_{k}} = x_{n - k + 2},\ 2 \leqslant k \leqslant n \}.$$ 
    In the light of \eqref{vomega}, notice that
    \[ g_{n-k+2} = v_\omega^{n-k+1} = \overline{v_\omega^{k-1}} = \overline{f_k},\ 2 \leqslant k \leqslant n. \]
    As such, the vector $Gx$ is real if and only if $x \in \mathbb{C}_{\rm sym}^n$.
\end{remark}

%--------------
\begin{example}
    In view of the inequalities given by \eqref{circulant_inequalities} and \thref{rem_symm}, if $x \in \mathbb{C}^2$ and $\Lambda(x) = \{ x_1, x_2 \}$, then $\Lambda$ is realizable by a circulant matrix if and only if $x_1, x_2 \in \mathbb{R}$ and 
    \begin{align*}
        \left\{ 
        \begin{array}{l}
            x_1 + x_2 \geqslant 0  \\
            x_1 - x_2 \geqslant 0  
        \end{array} \right..
    \end{align*}

    If $x \in \mathbb{C}^4$ and $\Lambda(x) = \{ x_1, x_2, x_3, x_4 \}$, then $\Lambda$ is realizable by a circulant matrix if and only if
    $x_1, x_3 \in \mathbb{R}$, $x_4 = \overline{x_2}$, and
    \begin{align*}
        \left\{ 
        \begin{array}{*{9}{l}}
            x_1 & + & x_2  & + & x_3 & + & \overline{x_2} & \geqslant & 0  \\
            x_1 & - & x_2i & - & x_3 & + & \overline{x_2}i & \geqslant & 0  \\
            x_1 & - & x_2 & + & x_3 & - & \overline{x_2} & \geqslant & 0  \\
            x_1 & + & x_2i & - & x_3 & - & \overline{x_2}i & \geqslant & 0  
        \end{array} \right..
    \end{align*}
\end{example}

%--------------
\begin{theorem}
    If $m,n \in \mathbb{N}$, then $F_m \otimes F_n$ is ideal, extremal, and $\dim \mathcal{C}(F_m \otimes F_n) = mn$. 
\end{theorem}

\begin{proof}
    Immediate from \thref{kronideal} and \thref{dftconetope}.
\end{proof}

%--------------
\begin{theorem}
    [half-space description]
    If $m,n \in \mathbb{N}$, then
    $$\mathcal{C}(F_m \otimes F_n) = \bigcap_{k \in \bracket{n}} \mathsf{H}_k,$$
    where 
    $$\mathsf{H}_k \coloneqq \mathsf{H}({f_k}) \cap \mathsf{H}(\ii {f_k}) \cap \mathsf{H}(-\ii {f_k})$$
    and $f_k \coloneqq (F_m \otimes F_n) e_k$, $k \in \bracket{mn}$.
\end{theorem}

%------------------
\begin{proposition}
    If $m, n \in \mathbb{N}$, then:
        \begin{enumerate}
            [label=(\roman*)]
            \item $\mathcal{C}({F_m}\otimes{F_n}) = \mathcal{C}\left(\overline{F_m}\otimes\overline{F_n}\right)$;
            \item $\overline{F_m} \otimes \overline{F_n}$ is ideal; and 
            \item $\mathcal{P}(\overline{F_m} \otimes \overline{F_n}) = \mathcal{P}_r(\overline{F_m} \otimes \overline{F_n})$.
        \end{enumerate}
\end{proposition}

\begin{proof}
    Analogous to the proof of \thref{propF} using \eqref{FmFninv} and properties of the Kronecker product.  
\end{proof}

%----------------
\begin{corollary}
    If $x\in \mathbb{C}^n$, then $\Lambda(x)$ is realizable by an $(m,n)$ block-circulant matrix with circulant blocks if and only if $(F_m \otimes F_n) x \geqslant 0$.
\end{corollary}

%--------------
\begin{theorem}
    If $k \in \mathbb{N}_0$ and $n \in \mathbb{N}$, then $F_n \otimes H_{2^k}$ is ideal, extremal, and $\dim \mathcal{C}(F_n \otimes H_{2^k}) = 2^k n$
\end{theorem}

%----------------
\begin{corollary}
    If $x \in \mathbb{C}^n$, then $\Lambda(x)$ is realizable by a Klein block matrix with circulant blocks if and only if $(F_n \otimes H_{2^k})x \geqslant 0$.      
\end{corollary}

%-------------
\begin{remark}
    If $A \in \mathsf{M}_m(\mathbb{C})$ and $B \in \mathsf{M}_n(\mathbb{C})$, then there is a permutation matrix $P$ such that $A \otimes B = P(B \otimes A) P^\top$ \cite[Corollary 4.3.10]{hj1994}. As a consequence, $F_m \otimes F_n \sim F_n \otimes F_m$. Furthermore, if $m$ and $n$ are relatively prime, then there are permutation matrices $P$ and $Q$ such that $F_{mn} = P (F_m \otimes F_n) Q$ \cite{g1958,r1980}. Thus, $F_{mn} \sim F_m \otimes F_n$ whenever $\gcd(m,n) = 1$. It is also clear that $F_{mn} \not\sim F_m \otimes F_n$ whenever $\gcd(m,n) > 1$.
\end{remark}

Recall that if $A_1,\ldots, A_m$ are matrices with $A_k \in \mathsf{M}_{m_k \times n_k}(\mathbb{C})$, then
\[  
\bigotimes_{k=1}^m A_k 
\coloneqq
\begin{cases}
    A_1,                                                    & k = 1 \\
    \left(\bigotimes_{k=1}^{m-1} A_k \right) \otimes A_m,   & k > 1.
\end{cases}
\]

%--------------
\begin{theorem}
    If $S_1,\ldots,S_m$ are Perron similarities with $S_k \in \mathsf{GL}_{n_k}(\mathbb{C})$, then $\bigotimes_{k=1}^m S_k$ is a Perron similarity. 
\end{theorem}

\begin{proof}
    Follows by a straightforward proof by induction on $m$ in conjunction with \thref{kronPS}. 
\end{proof}

%--------------
\begin{theorem}
    If $S_1,\ldots,S_m$ are ideal with $S_k \in \mathsf{GL}_{n_k}(\mathbb{C})$, then $\bigotimes_{k=1}^m S_k$ is ideal. 
\end{theorem}

\begin{proof}
    Follows by a straightforward proof by induction on $m$ in conjunction with \thref{kronideal}. 
\end{proof}

%--------------
\begin{corollary}
    If 
    \[ S \coloneqq \left( \bigotimes_{j=1}^N F_{n_j} \right) \otimes H_{2^k},\ k \in \mathbb{N}_0,\ n_j \in \mathbb{N},\ j \in\bracket{N}, \]
    then  $S$ is ideal and extremal.
\end{corollary}

%--------------
\begin{example}
    The matrices $F_{24}$, $H_2 \otimes F_{12}$, $H_4 \otimes F_6$, $H_8 \otimes F_3$, and $F_4 \otimes F_{6}$ are ideal and extremal Perron similarities of order $24$.
\end{example}

%---------------------------------------------------------------------
\section{Geometrical representation of the spectra of 4-by-4 matrices}

% If $n=2$, then the matrix $F_2$ yields all possible spectra and if $n=3$, then the matrices $F_2 \oplus F_1$ and $F_3$ yield all possible spectra. 

% For details of Perron similarities that yield all possible spectra when $n=4$ and $\lambda_1$, $\lambda_2$, $\lambda_3$, and $\lambda_4$ are real, see \cite{jp2016}.

The problem of finding a geometric representation of all vectors $\begin{bmatrix} \lambda & \alpha & \omega \end{bmatrix}^\top$ in $\mathbb{R}^3$ such that $\{1,\lambda,\alpha+\omega\ii,\alpha-\omega\ii \}$ is the spectrum of a 4-by-4 nonnegative matrix (we denote this region by $\mathbb{B}$) was posed by Egleston et al.~\cite[Problem 1]{eln2004}. 

In 2007, Torre-Mayo et al.~\cite{tmetal2007} characterized the coefficients of the characteristic polynomials of four-by-four nonnegative matrices and in 2014, Benvenuti \cite{b2014} used these results to produce the region given in Figure \ref{fig:ben}. It is worth noting here that this approach is not applicable to any other dimension.
\begin{figure}[H]
    \centering
    \includegraphics[width = 0.333\textwidth]{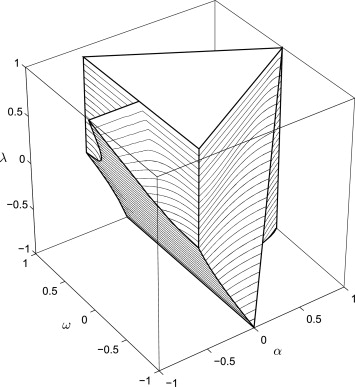}
    \caption{Geometrical representation of the spectra of four-by-four matrices by Benvenuti \cite[Figure 11]{b2014}.}
    \label{fig:ben}
\end{figure} 

%--------------------------------------------
\subsection{Region Generated by Spectratopes}

%------------
\begin{lemma}
    If 
    \[ S = 
    \begin{bmatrix}
        1 & e_1^\top \\
        0 & F_{n}
    \end{bmatrix} \in \mathsf{GL}_{n+1}(\mathbb{C}), \]
    then 
    \[ S^{-1} = 
    \begin{bmatrix}
        1 & -e^\top/n    \\
        0 & F_{n}^{-1}
    \end{bmatrix}. \]
\end{lemma}

\begin{proof}
    Notice that 
    \begin{align*}
        \begin{bmatrix}
            1 & e_1^\top \\
            0 & F_{n}
        \end{bmatrix}
        \begin{bmatrix}
            1 & -e^\top/n    \\
            0 & F_{n}^{-1}
        \end{bmatrix} 
        =
        \begin{bmatrix}
            1 & -e^\top/n + e_1^\top F_{n}^{-1}    \\
            0 & I_n
        \end{bmatrix}
    \end{align*}
    and since $F_n^{-1} = \frac{1}{n}\overline{F_n}$, it follows that $e_1^\top F_{n}^{-1} = e^\top/n$ and the result follows.
\end{proof}

%--------------
\begin{theorem}
    \thlabel{cartprodFn}
    If 
        \[ S = 
        \begin{bmatrix}
        1 & e_1^\top \\
        0 & F_{n}
        \end{bmatrix}, \]
    then $\mathcal{P}(S) = [0,1] \times \mathcal{P}(F_{n}) = [0,1] \times \mathcal{P}_r(F_{n})$. Furthermore, $S$ is extremal.
\end{theorem}

\begin{proof}
    Let $x \in \mathbb{C}^{n+1}$ and let $y = \pi_1(x) \in \mathbb{C}^n$. Since 
    $$e_1^\top D_y F_n^{-1} = \frac{y_1}{n} e^\top,$$ 
    it follows that
    \begin{align*}
        S D_x S^{-1}
        &=
        \begin{bmatrix}
            1 & e_1^\top \\
            0 & F_{n}
        \end{bmatrix}
        \begin{bmatrix}
            x_1 & 0 \\
            0   & D_y
        \end{bmatrix}
        \begin{bmatrix}
            1 & -e^\top/n    \\
            0 & F_{n}^{-1}
        \end{bmatrix} \\
        &=
        \begin{bmatrix}
            x_1 & -\frac{x_1}{n} e^\top + e_1^\top D_y F_n^{-1} \\
            0 & F_n D_y F_n^{-1}
        \end{bmatrix} \\
        &= 
        \begin{bmatrix}
            x_1 & \frac{y_1 - x_1}{n} e^\top \\
            0 & F_n D_y F_n^{-1}
        \end{bmatrix}.
    \end{align*}

    If $x \in \mathcal{P}(S)$, then the matrix above is stochastic. Thus, $x_1 \in [0,1]$ and $y \in \mathcal{P}(F_n) = \mathcal{P}_r(F_n)$, i.e., $x \in [0,1] \times \mathcal{P}_r(F_{n})$.

    If $x \in [0,1] \times \mathcal{P}_r(F_{n})$, then $y \coloneqq \pi_1(x) \in \mathcal{P}_r(F_{n}) = \mathcal{P}(F_{n})$. Since the first column of $F_n e_1 = e$, it follows that $y_1 = 1$ and the matrix 
    $$ SD_x S^{-1} = 
    \begin{bmatrix}
        x_1 & \frac{1 - x_1}{n} e^\top \\
        0 & F_n D_y F_n^{-1}        
    \end{bmatrix}
    $$
    is clearly stochastic, i.e., $x \in \mathcal{P}(S)$. 

    Finally, note that $S$ is extremal because $F_n$ is extremal.
\end{proof}

If $S \in \mathsf{GL}_{4}(\mathbb{C})$ is a Perron similarity, then   
\[ S \sim \begin{bmatrix}
1   & 1           & 1                       & 1                         \\
1   & \lambda_2   & \alpha_2 + \omega_2\ii  & \alpha_2 - \omega_2\ii    \\
1   & \lambda_3   & \alpha_3 + \omega_3\ii  & \alpha_3 - \omega_3\ii    \\
1   & \lambda_4   & \alpha_4 + \omega_4\ii  & \alpha_4 - \omega_4\ii    
\end{bmatrix}. \]
Furthermore, if $S$ is ideal and $x \in \mathcal{P}_r(S)$, then there are nonnegative scalars $\gamma_1$, $\gamma_2$, $\gamma_3$, and $\gamma_4$ such that $\sum_{i=1}^4 \gamma_i = 1$ and 
\begin{align*}
x 
&= 
\begin{bmatrix}
1 & 
\gamma_1 + \sum_{i=2}^4 \gamma_i \lambda_i                              & 
\gamma_1 + \sum_{i=2}^4 \gamma_i\left(\alpha_i + \omega_i\ii \right)    & 
\gamma_1 + \sum_{i=2}^4 \gamma_i\left(\alpha_4 - \omega_4\ii \right)   
\end{bmatrix}                                                                   \\
&= 
\begin{bmatrix}
1 & 
\gamma_1 + \sum_{i=2}^4 \gamma_i \lambda_i                                      & 
\gamma_1 + \sum_{i=2}^4 \gamma_i\alpha_i + \sum_{i=2}^4 \gamma_i \omega_i\ii    & 
\gamma_1 + \sum_{i=2}^4 \gamma_i\alpha_i - \sum_{i=2}^4 \gamma_i\omega_i\ii     
\end{bmatrix}.  
\end{align*} 
Consequently, 
\[ 
\begin{bmatrix}
\gamma_1 + \sum_{i=2}^4 \gamma_i \lambda_i  & 
\gamma_1 + \sum_{i=2}^4 \gamma_i\alpha_i    & 
\sum_{i=2}^4 \gamma_i \omega_i 
\end{bmatrix} \in \mathbb{B}
\]
and 
\[ 
\begin{bmatrix}
\gamma_1 + \sum_{i=2}^4 \gamma_i \lambda_i  & 
\gamma_1 + \sum_{i=2}^4 \gamma_i\alpha_i    & 
-\sum_{i=2}^4 \gamma_i \omega_i 
\end{bmatrix} \in \mathbb{B}.
\]

When $n=4$, the K-arcs in the upper-half region are:
\begin{itemize}
    \item $K_4(0,1)$ (Type 0);
    \item $K_4(1/4, 1/3)$ (Type I); and
    \item $K_4(1/3,1/2)$ (Type II).
\end{itemize}
However, it is known that $K_4(1/3,1/2) = \overline{K_4^2(1/4, 1/3)}$ \cite[Remark 5.3]{mnp2024}. As mentioned earlier, the Type 0 arc is subsumed in the Type I arc. Thus, it suffices to consider Perron similarities of realizing matrices corresponding to the arc $K_4(1/4, 1/3)$. Figure \ref{fig:dft4} depicts the projected spectratope corresponding to $F_4$.
    \begin{figure}[H]
        \centering
            \includegraphics[width=0.5\linewidth]{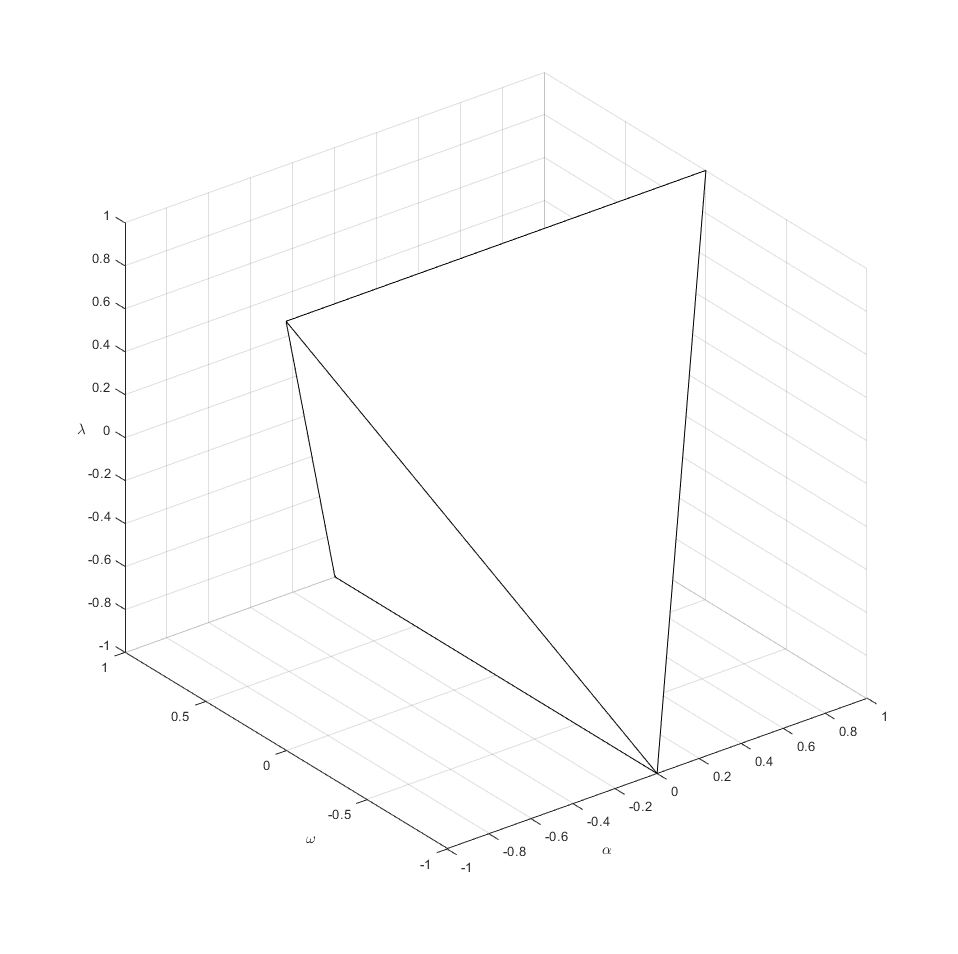}
            \caption{Projected spectratope of $F_4$.}
            \label{fig:dft4}
    \end{figure}
Figure \ref{fig:sub1} contains spectra derived from the projected spectratopes of these Peron similarities. Notice that Figure \ref{fig:sub2} matches the Karpelevich region when $n=4$.

    \begin{figure}[H]
        \centering
            \includegraphics[width=.5\textwidth]{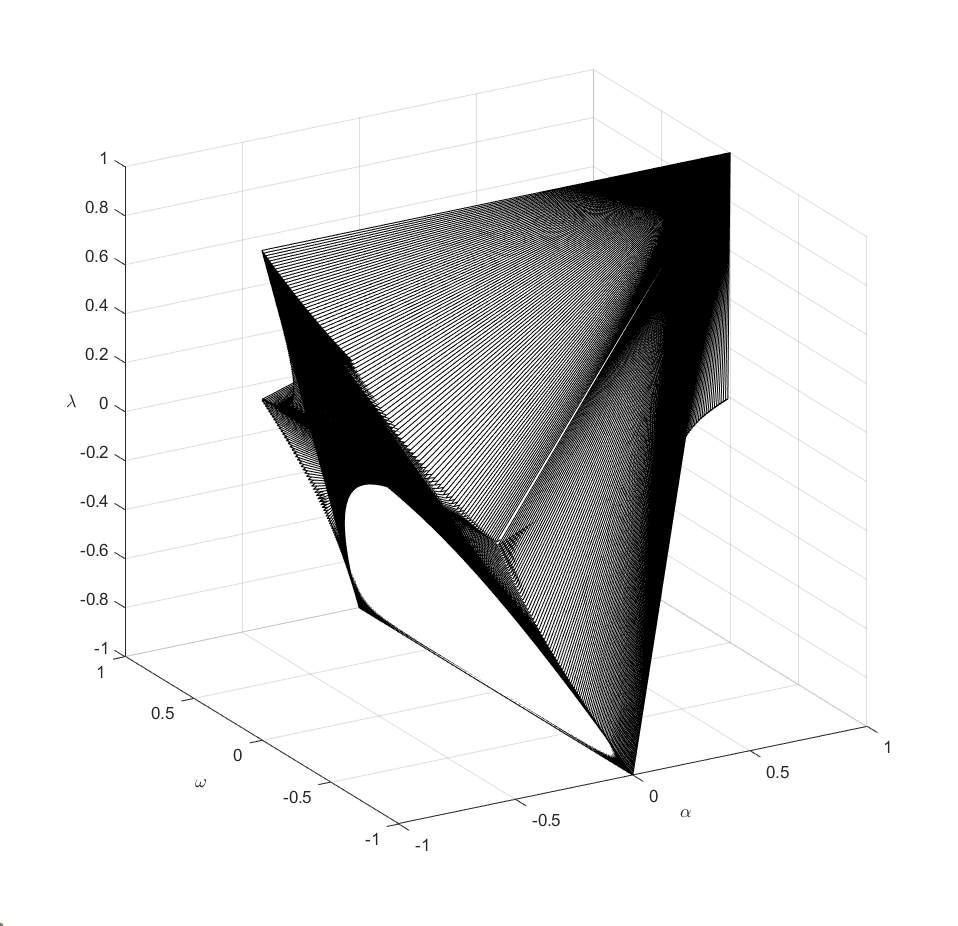}
            \caption{Projected spectratopes of Peron similarities arising from $K_4(0,1)$ and $K_4(1/4, 1/3)$.}
            \label{fig:sub1}
    \end{figure}
        
    \begin{figure}[H]
        \centering
            \includegraphics[scale=.250]{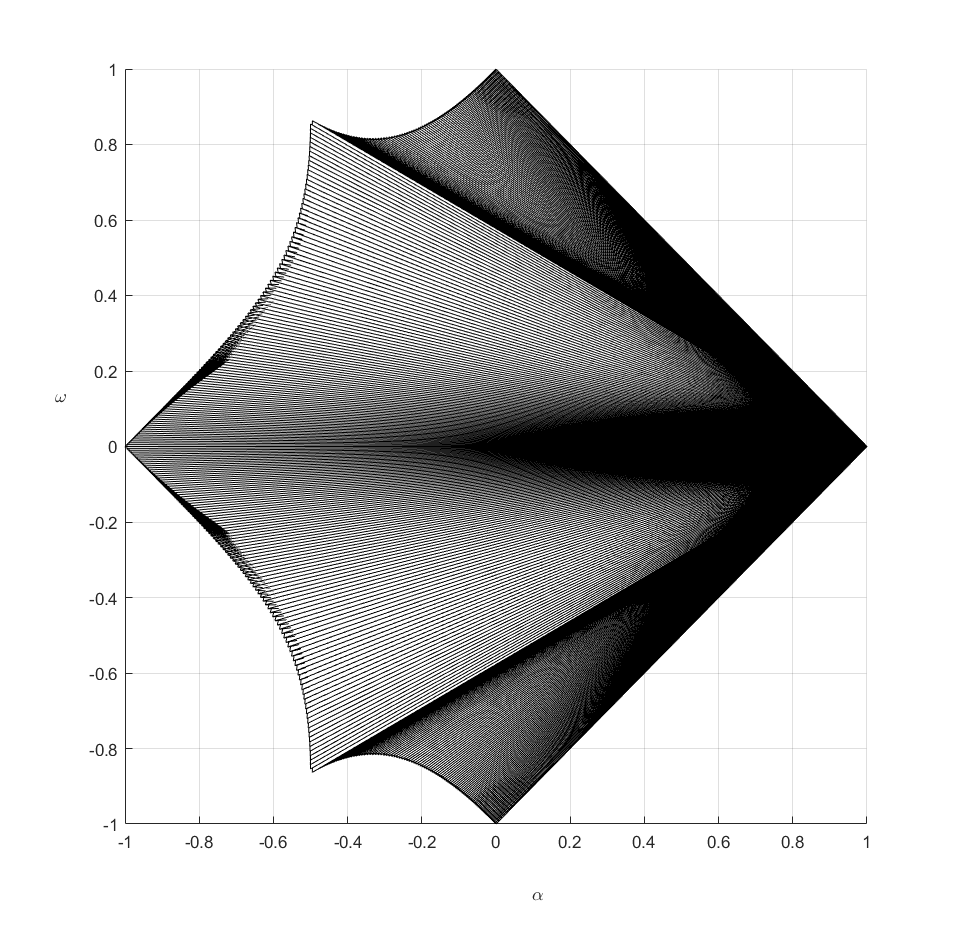}
            \caption{The $\alpha\omega$-view of Figure \ref{fig:sub1} matching $\Theta_4$.}
            \label{fig:sub2} 
    \end{figure}

If 
$S = 
\begin{bmatrix}
    1 & e_1^\top \\
    0 & F_3
\end{bmatrix} \in \mathsf{GL}_4(\mathbb{C}),$
then $\mathcal{P}(S) = [0,1] \times \mathcal{P}_r(F_3)$ by \thref{cartprodFn}. Figure \ref{fig:specfinal} adds the projected spectratope of $S$.
\begin{figure}[H]
    \includegraphics[width = .5\textwidth]{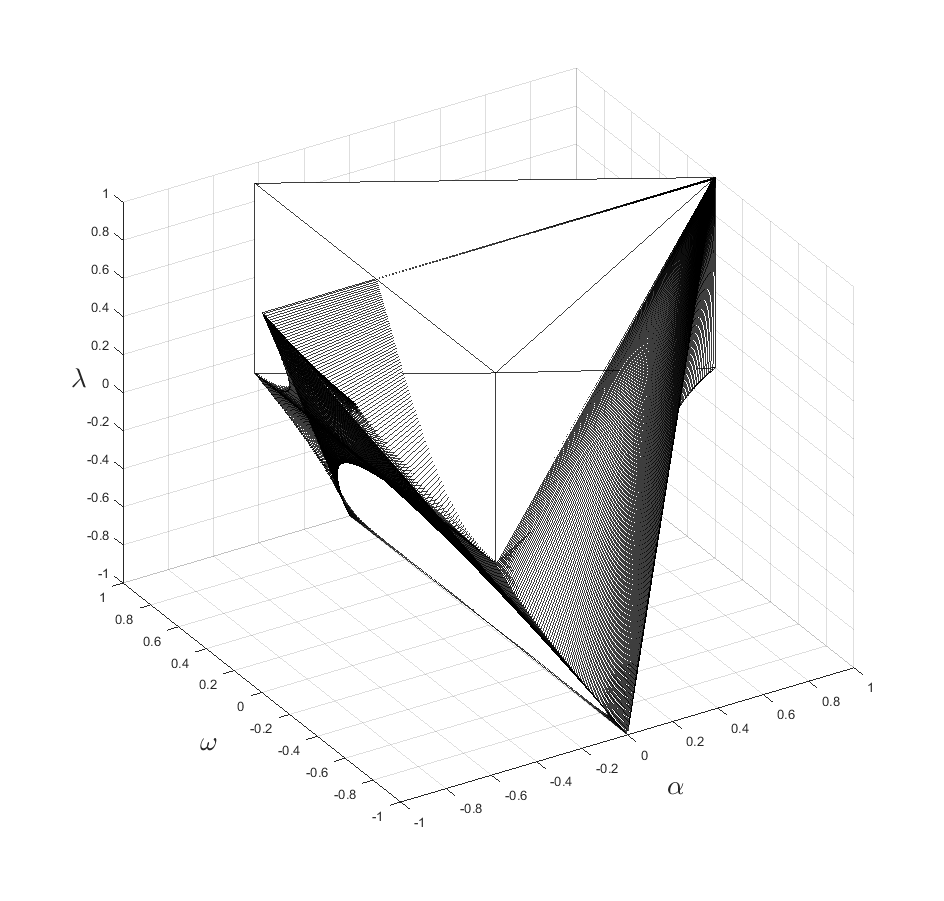}
    \caption{Geometric representation of the spectra of $4$-by-$4$ matrices via spectratopes.}\label{fig:specfinal}
\end{figure} 

Notice that the missing region, which is small relative to the entire region, contains spectra such that $-1 \leqslant \lambda \leqslant 0$, $0 \leqslant \alpha \leqslant 1$, and $-1 \leqslant \omega \leqslant 0$. 

%-----------------------------------------
\section{Implications for Further Inquiry}

Theorems \ref{hadamardcones} and \ref{CSpolyconePSpolytope} demonstrate that $\mathcal{C}(S)$ ($\mathcal{P}(S)$) is a polyhedral cone (polytope) that is closed with respect to the Hadamard product. As such we pose the following. 

%---------------
\begin{question}
    If $K$ is a polyhedral cone (polytope) that is closed with respect to the Hadamard product, is there an invertible matrix $S$ such that $K = \mathcal{C}(S)$ ($K = \mathcal{P}(S)$)?
\end{question}

The following conjecture, which fails when $n = 2$ and $n = 3$, would demonstrate that characterizing the extreme points is enough to characterize $\mathbb{SL}_1^n$. 

%-----------------
\begin{conjecture}
    If $n > 3$, then $\partial \mathbb{SL}_1^n \subseteq \mathbb{E}^n$, i.e., points on the boundary are extremal.
\end{conjecture}

\bibliographystyle{amsplain}
\bibliography{master}

\end{document}